\documentclass[11pt, reqno]{amsart}
\usepackage[utf8]{inputenc}
\usepackage{amsfonts,amssymb,amsmath,amsthm}
\usepackage[margin=1in]{geometry}
\usepackage[hidelinks]{hyperref} \usepackage{cleveref} \usepackage{enumitem} \usepackage{microtype} \usepackage{mathtools} \usepackage{color} \usepackage{tikz-cd} \usepackage{comment} \usepackage{graphicx} 
\usepackage[numbers,sort&compress]{natbib} 

\newcommand{\nocontentsline}[3]{}
\newcommand{\tocless}[2]{\bgroup\let\addcontentsline=\nocontentsline#1{#2}\egroup}

\def\MR#1{}

\def\Z{\mathbb{Z}}
\def\Q{\mathbb{Q}}
\def\F{\mathbb{F}}

\theoremstyle{plain}
\newtheorem{theorem}{Theorem}[section]
\newtheorem{corollary}[theorem]{Corollary}
\newtheorem{lemma}[theorem]{Lemma}
\newtheorem{proposition}[theorem]{Proposition}

\theoremstyle{definition}
\newtheorem{definition}[theorem]{Definition}

\theoremstyle{remark}
\newtheorem{remark}[theorem]{Remark}

\def\Z{\mathbb Z}

\def\Q{\mathbb Q}
\def\R{\mathbb R}
\def\C{\mathbb C}
\def\F{\mathbb F}

\def\cB{\mathcal B}
\def\cA{\mathcal A}
\def\cP{\mathcal P}
\def\cF{\mathcal F}

\def\Z{\mathbb Z}

\def\Q{\mathbb Q}
\def\R{\mathbb R}
\def\C{\mathbb C}
\def\F{\mathbb F}

\def\cB{\mathcal B}
\def\cA{\mathcal A}

\def\cP{\mathcal P}
\def\cG{\mathcal G}

\def\mod{\operatorname{mod}}

\def\GSp{\operatorname{GSp}}

\def\log{\operatorname{log}}

\def\sym{\operatorname{sym}}

\title[Joint Sato-Tate Laws]{Joint Sato-Tate Laws for Transformations of Hecke Eigenvalues: The Vertical Case 
}
\author[Hamdar]{Mohammad H. Hamdar}
\address{Department of Mathematics \&
Statistics, Concordia University, Montreal, Quebec, Canada\\
And\
D\'epartement de math\'ematiques et de statistique\\
	Universit\'e de Montr\'eal\\
	Montr\'eal, Quebec\\
	Canada}
\email{mohammadhussein.hamdar@mail.concordia.ca}
\title[Joint Sato-Tate Laws]{Joint Sato-Tate Laws for Transformations of Hecke Eigenvalues: The Vertical Case 
}
\date{}
\author[Wang]{Tian Wang}
\address{Department of Mathematics \&
Statistics, Concordia University, Montreal, Quebec, Canada}
\email{tian.wang@concordia.ca}
\date{}


\begin{document}

\maketitle

\begin{abstract}
    We introduce a framework for studying joint equidistribution problems with effective error terms. Our approach involves proving a higher-dimensional, $\mu$-analogue of the Erd\"{o}s-Tur\'{a}n inequality, and utilizing the theory of the Hardy-Krause (H-K) variation from analysis. In particular, we develop a method to approximate a broad class of functions by ones of bounded H-K variation.  As applications, we study vertical Sato-Tate problems for spaces of cusp forms and families of elliptic curves over finite fields. This leads to new effective equidistribution results on arithmetic relations and, more generally, on multivariate functions of Fourier coefficients and Frobenius traces. 
\end{abstract}

\tableofcontents

\section{Introduction and Main Results}
Let $S_k(N)$ denote the space of cusp forms of weight $k$ and level $\Gamma_0(N)$ and let $p$ be a prime number. For a Hecke eigenform $f\in S_k(N)$, its Fourier coefficients $a_p(f)$ are the eigenvalues of the Hecke operators $T_p(N,k)$. From Deligne's seminal work \cite{Deligne}, we know that the eigenvalues $a_p(f)$ lie in the interval $[-2p^{\frac{k-1}{2}},2p^{\frac{k-1}{2}}]$. If we fix $f$ and vary the prime $p$, it is interesting to ask how the normalized eigenvalues 
$$\tilde{a}_p(f)=\frac{a_p(f)}{2p^{\frac{k-1}{2}}}$$
are distributed in $[-1,1]$. This is predicted by the Sato-Tate conjecture, which states that $\tilde{a}_p(f)$ (for $f$ non-CM) is equidistributed in $[-1,1]$ with respect to the Sato-Tate measure
\[\mu_{ST}(x):=\frac{2}{\pi}\sqrt{1-x^2}\,dx.\] This conjecture was recently settled after a considerable amount of work by many people, culminating in the papers \cite{Clozel, Taylor, MR2827723}. 

Instead of fixing an eigenform $f$ and varying over primes $p$, we can also fix a prime $p$ and vary over the forms $f_i$, $1\leq i\leq d$ as $d\to \infty$, where $d$ is the dimension of the space $S_k(N)$ of size 
\begin{equation}\label{eq:dimension}
d\asymp kN\prod_{p|N}\left(1+\frac{1}{p}\right),
\end{equation}
and  $\{f_i\}_{1\leq i\leq d}$ is a Hecke eigenbasis of $S_k(N)$. This problem was first addressed by Sarnak in the context of Maass forms \cite{Sarnak87}, and by Serre \cite{Serre97} and  Conrey-Duke-Farmer \cite{MR1438595} in the context of classical modular forms. It is now known as the ``vertical" Sato-Tate problem. Denoting by $a_p(i)=a_p(f_i)$, the theorem of Serre states the following:

 Let $N, k$ be positive integers such that $k$ is even, and $p\nmid N$. For any $[\alpha,\beta]\subset [-1,1]$,  as $k+N\to \infty$, we have 
   \begin{equation}\label{Serre}
     \frac{1}{d}\#\left\{1\leq i\leq d:\frac{a_{p}(i)}{2p^{\frac{k-1}{2}}}\in [\alpha,\beta] \right\}\sim\int_{\alpha}^{\beta} \mu_p,
     \end{equation}
where
\begin{equation}\label{eqn: Plancheral}
  \mu_p(x):=\frac{2(p+1)}{\pi}\frac{\sqrt{1-x^2}}{(p^{\frac{1}{2}}+p^{-\frac{1}{2}})^2-4x^2}\,dx.
\end{equation}
The measure $\mu_p$ is the $p$-adic  Plancherel measure and it is the same measure Sarnak obtained in his analogous theorem for Maass forms \cite[Theorem 1.2]{Sarnak87}. This measure also appears in connection with eigenvalues of certain graphs, as it is the spectral measure of the nearest-neighbor Laplacian on a $p+1$ regular tree (see \cite{LPS} and \cite{Mckay}). It is worth noting that $\mu_p\to\mu_{ST}$ as $p\to \infty$, which explains the notation of $\mu_{\infty}$ for $\mu_{ST}$ by many authors. 

M. R. Murty and Sinha \cite{MurtySinha2009} improved Serre's result by obtaining an effective error term of $O\left(\frac{\log p}{\log (kN)}\right)$ in \eqref{Serre}. A joint version of \eqref{Serre} was then studied by Lau and Wang in \cite{LauWang}, where they also obtained an effective error.

 As for the ``horizontal" Sato-Tate results mentioned at the beginning (fixing a modular form or elliptic curve and varying over primes), an error term was first obtained by V. K. Murty \cite{KumarMurty} in the case of elliptic curves and by Rouse and Thorner \cite{ThornerRouse} in the case of modular forms. Both results require assuming the Generalized Riemann Hypothesis (GRH) for the symmetric power $L$-functions associated with the elliptic curves or modular forms.  A joint version for the elliptic curves case was initially posed as a question independently by Katz and Mazur (see \cite{Harris}), and was proved by Harris \cite{Harris}  in the two dimensional case (see also \cite{MR3012726}), with an error term later established by Bucur and Kedlaya \cite{MR3502938} under GRH for Rankin–Selberg convolutions of symmetric power $L$-functions associated to elliptic curves. A joint version for the modular forms case can be found in the work of Wong \cite{Wong} in the two dimensional case, with an unconditional effective error term obtained by Thorner \cite{Thornerjoint}.  


In this paper, we consider the distribution of arithmetic relations, and more generally, multivariable functions of Hecke eigenvalues, in a vertical setting. That is, we fix a prime and vary over $n$ spaces (or families), or we fix $n$ primes and vary over one space. We investigate this for Fourier coefficients of cusp forms, and for Frobenius traces of elliptic curves over finite fields. We will begin by presenting the results for modular forms.

\subsection{The Case of Modular Forms}

We start by defining a very general class of functions on $\R^n$ that our results work for. This includes most typical functions of interest. For example: continuous functions, piecewise-continuous functions, functions of bounded higher dimensional variations (see Appendix A for a piecewise-continuous function in $\R^2$ that is not of bounded Hardy-Krause variation, but is relevant to us).  
\begin{definition}\label{def: goodpair}
 Let $n\geq 1$ be an integer. Let $J\subset \R$ be a closed interval and $F$ a real valued function on $[-1,1]^n$. We say $(F,J)$ is a good pair if the following two conditions hold: 
  \begin{itemize}
 \item[(1)] $F^{-1}(J)$ is Lebesgue measurable; 
  \item[(2)] For any $l_1,\dots,l_n\in\Z_{>0}$, given a partition of $[-1,1]^n$ by $\prod_{j=1}^n l_j$ boxes $B$,  each of size $\frac{2}{l_1}\times \frac{2}{l_2}\times\cdots \times\frac{2}{l_n}$, we have that both $F^{-1}(J)$ and $[-1,1]^n\backslash F^{-1}(J)$ can be partitioned by $O_{F}\left(\sum_{j=1}^n l_j\right)$-many larger (or equal) boxes composed of the original boxes $B$.
  \end{itemize}
\end{definition}
Condition (2) of \Cref{def: goodpair} is somewhat technical, but it turns out this is the least sufficient condition we need to make \Cref{lem:variation-h} and the proceeding arguments in \Cref{subsection:error-term} work. Note that if we can partition $F^{-1}(J)$ and $[-1, 1]^n\backslash F^{-1}(J)$ with less boxes, then the error terms in our theorems improve. That is, for pairs $(F,J)$ in \Cref{def: goodpair} with better partitions than condition (2), the errors become smaller upon optimization (see \Cref{subsection:error-term}). We again emphasize that any typical function of interest satisfies condition (2) for any closed interval $J\subset \R$. 

To state our first theorem, let $S_{k_j}(N_j)$, $1 \leq j \leq n$, be spaces of cusp forms of even weights $k_j \geq 2$ and levels $N_j$. Throughout this paper, all weights are assumed to be even. For each $1\leq j\leq n$, let $\{f_{i_j}\}_{1\leq i_j\leq d_j}$ be a Hecke eigenbasis for the respective $d_j$-dimensional space. For a prime $p\nmid N_j$,  let $a_{p}(i_j)$ denote the $p$-th Fourier coefficient of the form $f_{i_j}$. Further, denote by 
\begin{equation}\label{eqn: Hp}
H_{p}(x):=\frac{2(p+1)}{\pi}\frac{\sqrt{1-x^2}}{(p^{\frac{1}{2}}+p^{-\frac{1}{2}})^2-4x^2}.
\end{equation}

 
\begin{theorem}\label{thm:main-thm}
Let $n\geq 2$ be a fixed positive integer, and  $(F,J)$ a good pair in the sense of \Cref{def: goodpair}.
\begin{enumerate}
    \item[(i)] Fix a prime $p$. For each $1\leq j\leq n$,  consider $S_{k_j}(N_j)$ of dimension $d_j$ with  $p\nmid N_j$. We have that for any $\epsilon>0$, as $k_1+N_1,\ldots,k_n+N_n\to \infty$
 \begin{align*}
     &\frac{1}{d_1d_2\dots d_n} \#\left\{1\leq i_1\leq d_1,\dots, 1\leq i_n\leq d_n: F\left(\frac{a_{p, 1}(i_1)}{2p^{\frac{k_1-1}{2}}},\dots, \frac{a_{p,n}(i_n)}{2p^{\frac{k_n-1}{2}}}\right)\in J \right\}\\
     &\hspace{5cm}=\int_{\substack{F(x_1, \ldots, x_n)\in J\\(x_1, \ldots, x_n)\in [-1, 1]^n}}  \ \mu_{p, n} +O_{n, F,\epsilon}\left(\left(\frac{\log p}{\log \left(d_1\cdots d_n\right)}\right)^{1-\epsilon}\right),
\end{align*} 
 where 
\begin{equation}\label{mu_p,n}
\mu_{p,n}:=\mu_{p, n}(x_1, \ldots, x_n)=H_p(x_1)\dots H_p(x_n)\,dx_1\dots\,dx_n,
\end{equation}
with $H_p(x_i)$  defined in \eqref{eqn: Hp}. 
\item[(ii)] Fix $n$ primes $p_1,\dots,p_n$ and let  $S_k(N)$ be a space of dimension $d$ 
with $p_1,\dots,p_n\nmid N$. We have that for any 
$\epsilon>0$, as $k+N\to \infty$
 \begin{align*}
     &\frac{1}{d} \#\left\{1\leq i\leq d:\, F\left(\frac{a_{p_1}(i)}{2p_1^{\frac{k-1}{2}}},\dots, \frac{a_{p_n}(i)}{2p_n^{\frac{k-1}{2}}}\right)\in J \right\}\\
     &\hspace{3cm}=\int_{\substack{F(x_1, \ldots, x_n)\in J\\(x_1, \ldots, x_n)\in [-1, 1]^n}}  \ \mu_{p_1}(x_1)\dots\mu_{p_n}(x_n) +O_{n, F,\epsilon}\left(\left(\frac{\log (p_1\cdots p_n)}{\log d}\right)^{1-\epsilon}\right),
\end{align*} 
 where each $\mu_{p_i}$ is a Plancherel measure defined in \eqref{eqn: Plancheral}.
\end{enumerate}
\end{theorem}

\begin{remark}
Denote by $\lambda$ the Lebesgue measure on $\R^n$. 
 We note that if $ \lambda(F^{-1}(J))>0,$ then 
\[
\int_{\substack{F(x_1, \ldots, x_n)\in J\\(x_1, \ldots, x_n)\in [-1, 1]^n}}   \, \mu_{p,n}\asymp_{n}  1,
\] 
which means it is indeed the main term of part (i) in the theorem. And similarly for the main term of part (ii). Otherwise,  we get upper bounds in \Cref{thm:main-thm}. 
\end{remark}
 For a similar statement in the one dimensional case $n=1$, we refer the reader to  \Cref{thm:nforms}  in \Cref{sec:Joint-distribution}. Note that the error term in the $n=1$ case is better by an exponent of $\epsilon$. One can try to get a similar error term for higher dimensions, i.e. removing the $\epsilon$ in the power, by making the dependency in the implied constant on $n$ explicit, which is possible if one is more careful, and then following the argument at the end of \Cref{subsection:main-term}.

The same statements can be obtained when restricting only to spaces of newforms. This follows by now using the trace formula for newforms in \cite[Theorem 5]{MurtySinha-newform} instead of the Eichler-Selberg trace formula that gives us \Cref{lem:effective_trace}  and \Cref{lem:n-case}. 

We remark that ``asymptotic independence" as in \Cref{thm:nforms} does not directly yield results like those in \Cref{thm:main-thm} using probability theory, and one reason is because the error terms in the former theorem are independent of the fixed intervals, which makes the asymptotics blow up.

We now give two immediate corollaries to the theorem above. 
Let $0\neq \lambda_j\in \R$ for any $1\leq j\leq n$ and $n\geq 2$, and denote by
\begin{equation}\label{eq: measure hp}
h_{p, j}(x):=\frac{H_p\left(x/\lambda_j\right) }{|\lambda_j|}
\end{equation}
and
\begin{equation}\label{eq: measure hp2}
h_{p_j}(x):=\frac{H_{p_j}\left(x/\lambda_j\right) }{|\lambda_j|}.
\end{equation}

\begin{corollary}\label{cor:main-cor}
 Fix a prime $p$,  a positive integer $n\geq 2$, and   a closed interval $J\subseteq \R$. For each $1\leq j\leq n$, consider  $S_{k_j}(N_j)$ of dimensions $d_j$ such that $p\nmid N_j$, and let $\lambda_j$ be nonzero real numbers. Then, for any $\epsilon>0$, as $k_1+N_1,\dots,k_n+N_n\to\infty$
  \begin{align*}    &\frac{1}{d_1d_2\dots d_n} \#\left\{1\leq i_1\leq d_1,\dots, 1\leq i_n\leq d_n: \sum_{1\leq j\leq n}\lambda_j\left(\frac{a_{p, j}(i_j)}{2p^{\frac{k_j-1}{2}}}\right)\in J \right\}\\&\hspace{6cm}=\int_{J}  \ \mu_{p, n}^* +O_{n,\epsilon, \lambda_1, \ldots, \lambda_n} \left(\frac{\log p}{\log \left(d_1\cdots d_n\right)}\right)^{1-\epsilon},\end{align*}  where $\mu_{p, n}^*:=h_{p, 1}*\ldots *h_{p,n}(x) \, d x$ with $h_{p,j}$ defined in \eqref{eq: measure hp} and $*$ is the convolution of functions. Consequently,  assuming $k_1=\ldots=k_n$, we obtain for any $t\in \R$, \begin{align}\label{eq:Lang-Trotter-modular-form}
     &\frac{1}{d_1d_2\dots d_n} \#\left\{1\leq i_1\leq d_1,\dots, 1\leq i_n\leq d_n: \sum_{1\leq j\leq n}\lambda_ja_{p, j}(i_j)= t\right\}\ll_{n,\epsilon, \lambda_1, \ldots, \lambda_n}\left(\frac{\log p}{\log \left(d_1\cdots d_n\right)}\right)^{1-\epsilon}.
\end{align} 

\end{corollary}

\begin{corollary}\label{cor:main-cor2}
 Fix a positive integer $n\geq 2$, primes $p_1,\dots,p_n$, and  a closed interval $J\subseteq \R$.   Consider $S_k(N)$ of dimension $d$ 
 with $p_1,\dots,p_n\nmid N$, and let $\lambda_1, \ldots, \lambda_n$ be nonzero real numbers. Then, for any $\epsilon>0$, as $k+N\to \infty$
  \begin{align*}    \frac{1}{d} \#\left\{1\leq i\leq d: \sum_{1\leq j\leq n}\lambda_j\left(\frac{a_{p_j}(i)}{2p_j^{\frac{k-1}{2}}}\right)\in J \right\}=\int_{J}  \ \mu_{p_1,\dots,p_n}^* +O_{n,\epsilon,\lambda_1, \ldots, \lambda_n} \left(\frac{\log (p_1\cdots p_n)}{\log d}\right)^{1-\epsilon},\end{align*}  where $\mu_{p_1,\dots,p_n}^*:=h_{p_1}*\ldots *h_{p_n}(x) \, d x$ with $h_{p_j}$ defined in \eqref{eq: measure hp2} and $*$ is the convolution of functions. 
\end{corollary}

Our main theorem provides effective equidistribution results for Hecke eigenvalues associated with certain Langlands functorial lifts, such as endoscopic and tensor product lifts. Let $f\in S_{k_1}(N_1)$ and  $g\in S_{k_2}(N_2)$ be newforms with trivial Nebentypus. Assume that   $f$ and $g$ are not multiples of each other and for all $p\mid (N_1, N_2)$, the Atkin-Lehner eigenvalues of $f$ and
$g$ coincide. Let $\mathcal{Y}(f, g)$ denote a Yoshida lift associated to the pair $(f, g)$ (see, e.g., \cite{MR586427, MR3365801}).  For all $p\nmid N_1N_2$,  the Hecke eigenvalue $\lambda_\mathcal{Y}(p)$ of  $\mathcal{Y}(f, g)$ satisfies 
$$
\lambda_\mathcal{Y}(p)=a_p(f)+a_p(g).
$$
Therefore, one may apply \Cref{thm:main-thm} to the sequence $\left(\tilde{a}_f(p), \tilde{a}_g(p)\right)_{f, g}$ and obtain the distribution of Hecke eigenvalues of Yoshida lifts of elliptic modular forms, with effective error terms. This offers a complementary perspective to the work of Kim, Wakatsuki, and Yamauchi on genuine Siegel cusp forms  $\GSp_{4}/\Q$ \cite[Theorem 4.1]{MR4125677}.  As another example, given  newforms $f\in S_{k_1}(N_1)$ and  $g\in S_{k_2}(N_2)$, we denote by $f\otimes g$ the tensor product lift of $f$ and $g$ (see, for instance, \cite{MR1792292} for the existence of such lifts). Then,  \Cref{thm:main-thm} can be applied with $F(X_1, X_2)=X_1X_2$ to give effective equidistribution results for Hecke eigenvalues of tensor product lifts of newforms.

Finally, we remark that our work can be used to generalize or provide vertical versions to previous results on Fourier coefficients of modular forms satisfying certain properties when considered in families. We mention here one example. The result in \eqref{eq:Lang-Trotter-modular-form} extends the work of Murty and Srinivas \cite{MR3544520} in connection to Maeda’s conjecture, where they obtained upper bounds for the number of pairs 
$(f, g)$ of newforms in $S_k(N)$ satisfying $F(\tilde{a}_p(f), \tilde{a}_p(g))=0$ with $F(X_1, X_2)=X_1-X_2$.

\subsection{The Case of Elliptic Curves}

Modularity establishes a one-to-one correspondence between $\Q$-isogeny classes of elliptic curves $E: y^2=x^3+ax+b$ over $\Q$ of conductor $N$, and normalized newforms $f_E \in S_2^{\text{new}}(N)$ with integer Fourier coefficients. Moreover, for each prime $p$, the Frobenius trace $$a_p(E):=p+1-\#E(\F_p)=-\sum_{x\mod p}\left(\frac{x^3+ax+b}{p}\right)$$ coincides with the $p$-th Fourier coefficient $a_p(f_E)$ of $f_E$. It is well known that $a_p(E)$ satisfies the Hasse bound $|a_p(E)| \le 2\sqrt{p}$.

In this setting, joint forms of the horizontal Sato-Tate conjecture for pairs of elliptic curves $(E_1, E_2)$ over $\mathbb{Q}$, namely, the distribution of pairs $(a_p(E_1), a_p(E_2))$ as $p$ varies, are known due to Harris \cite{Harris}, and an effective version due to Bucur-Kedlaya \cite{MR3502938}. However, there are very few (effective) results concerning the distribution for (non-trivial) transformations of $a_p(E_1)$ and $a_p(E_2)$. To the best of our knowledge, such considerations only appear implicitly in \cite{MR4819229, MR4735820}.

 For a prime $p\geq 5$, consider the family \[\cF_p=\{E_{a,b}:y^2=x^3+ax+b: \ a,b\in\F_p\ \text{and}\ 4a^3\neq-27b^2\} \]
of all elliptic curves over $\F_p$. This family has exactly $p(p-1)$ curves. Birch \cite{Birch68} studied the distribution of the normalized traces $\frac{a_p(E)}{2\sqrt{p}}$ as $p\to \infty$ and showed that they are equidistributed in $[-1,1]$ with respect to the Sato-Tate measure. Namely, he proved, as $p\to\infty$,
\begin{equation}\label{Birch0}
\frac{1}{p(p-1)}\#\left\{E\in\cF_p:\ \frac{a_p(E)}{2\sqrt{p}}\in [\alpha,\beta]\subset [-1,1] \right\}\sim \int_\alpha^\beta\,\mu_{ST}.
\end{equation}

He obtained this asymptotic by calculating all the moments of $a_p(E)$ over the family $\cF_p$. More precisely, he showed \cite[Theorem 1]{Birch68}\footnote{Note that the moments as written in \cite{Birch68} require further normalization, see \cite[Appendix B]{MillerMurty}. }
\[\frac{1}{p(p-1)}\sum_{E\in \cF_p}a_p(E)^{2R}\sim\frac{1}{R+1}\binom{2R}{R}p^{R}. \]
Using the methods of Niederreiter \cite{Niederreiter}, Banks and Shparlinski \cite{BanksShparlinski} made \eqref{Birch0} effective by obtaining an error of $O(p^{-1/4})$. Using a different method, Miller and Murty \cite{MillerMurty} also obtained an effective version of \eqref{Birch0}, but with a weaker error term.

Our main theorem will be the analog of the modular forms case in this context. That is, we study the distribution of general transformations acting on Frobenius traces of elliptic curves in $\cF_p$.

\begin{theorem}\label{thm:ellipticmain-thm}
Let $p$ be a prime, $n\geq 2$ be a positive integer, and  $(F,J)$ a good pair in the sense of \Cref{def: goodpair}. We have that for any $\epsilon>0$, as $p\to \infty$ 
 \begin{align*}
     &\frac{1}{p^n(p-1)^n} \#\left\{E_1,\dots,E_n\in \cF_p: F\left(\frac{a_{p}(E_1)}{2\sqrt{p}},\dots, \frac{a_{p}(E_n)}{2\sqrt{p}}\right)\in J \right\}\\
     &\hspace{5cm}=\int_{\substack{F(x_1, \ldots, x_n)\in J\\(x_1, \ldots, x_n)\in [-1, 1]^n}}  \ \mu_{ST}(x_1)\dots \mu_{ST}(x_n) +O_{n, F,\epsilon}\left(p^{-\frac{1}{4}+\epsilon}  \right).
\end{align*} 
\end{theorem}
For a similar statement in the one dimensional case, we refer the reader to \Cref{thm:fg} in \Cref{sec:FamilyAllCurves}. As in the modular forms case, the error term in the $n=1$ case is better by an exponent of $\epsilon$, and one can try to remove the $\epsilon$ in the power here by making the dependency in the implied constant on $n$ explicit.

We now give an immediate corollary to \Cref{thm:ellipticmain-thm}.
For $0\neq \lambda_j\in \R$, for any $1\leq j\leq n$ and $n\geq 2$, denote by
\begin{equation}\label{eq: measure kj}
k_{j}(x):=\frac{2}{\pi}\frac{\sqrt{1-\left(\frac{x}{\lambda_j}\right)^2}}{|\lambda_j|}.
\end{equation}

\begin{corollary}\label{cor:ellipticmain-cor}
 Let $p$ be a prime, $n\geq 2$ be a positive integer, and  $J\subset \R$ be a closed interval.   Let $\lambda_1, \ldots, \lambda_n$ be nonzero real numbers. Then, for any $\epsilon>0$, as $p\to\infty$ 
 \begin{align*}
     &\frac{1}{p^n(p-1)^n} \#\left\{E_1,\dots,E_n\in \cF_p:\sum_{1\leq j\leq n}\lambda_j\left(\frac{a_{p}(E_j)}{2\sqrt{p}}\right)\in J \right\}=\int_{J}  \ \mu_{ST, n}^* +O_{n,\epsilon,\lambda_1,\dots,\lambda_n}\left( p^{-\frac{1}{4}+\epsilon}  \right),
\end{align*} 
where $\mu_{ST, n}^*:=k_{ 1}*\ldots *k_{n}(x) \, dx$ and $*$ is the convolution of functions.
Consequently, we obtain for any $t\in \R$, 
\begin{align}\label{eq:Lang-Trotter-two-parameter}
     &\frac{1}{p^n(p-1)^n} \#\left\{E_1,\dots,E_n\in \cF_p: \sum_{1\leq j\leq n}\lambda_ja_{p}(E_j)= t\right\}\ll_{n,\epsilon,\lambda_1,\dots,\lambda_n}p^{-\frac{1}{4}+\epsilon} .
\end{align} 
\end{corollary}

The preceding theorem and corollary provide effective vertical Sato–Tate and Lang–Trotter type results for products of  $n$ elliptic curves. To place these results in context, recall that for an elliptic curve $E/\Q$, the Lang–Trotter conjecture \cite{MR568299} concerns the asymptotic behavior of the number of primes $p\leq x$ such that the Frobenius trace  $a_p(E)$ equals a fixed integer $t$.  This conjecture has since been generalized to general types of abelian varieties  $A$ over $\Q$ by Cojocaru,  Davis, Silverberg, and Stange \cite{MR3693659}, and to products of two elliptic curves $E_1, E_2$ over $\Q$ by Chen, Jones, and Serban \cite{MR4480293}, where they predict the asymptotic behavior for the number of primes $p\leq x$ for which $a_p(A)=t$ and $a_p(E_1\times E_2)=a_p(E_1)+a_p(E_2)=t$, respectively.  From this perspective, the direct horizontal analogue of \eqref{eq:Lang-Trotter-two-parameter} for products of $n$ elliptic curves was studied in \cite{MR4795351} in the special case $t = 0$.

Following the ideas of \cite{MR2383504, MR3624557, MR3395787}, we expect that these vertical asymptotic results can be leveraged to show that horizontal Sato-Tate and Lang-Trotter conjectures hold on average for products of elliptic curves. We plan to pursue this direction in future work. 

\subsection{Overview of the Paper} The first part of this paper is devoted to proving an Erd\"{o}s-Tur\'{a}n inequality for $n$-dimensional sequences in $[0,1]^n$ that works for generic measures $\mu$, and a general Koksma-Hlawka inequality that works for functions of bounded Hardy-Krause (H-K) variation. These are the contents of \Cref{thm:discrepancy} and \Cref{thm:general-KoksmaHlawka} of the paper. The former is obtained using the work of Cochrane \cite{Cochrane} on approximating characteristic functions in hypercubes using higher dimensional versions of Beurling-Selberg polynomials. And the latter is obtained using the theory of H-K variation and the work of G\"{o}tz \cite{MR1914223} on Haar-type wavelet analysis, together with \Cref{thm:discrepancy}. In the same section we also collect and prove several properties of H-K variation that we will need.

In Appendix A, we show that the function that we expect to give us \Cref{cor:main-cor} is in fact not of bounded variation in the sense of Hardy and Krause, and so it falls outside the scope of \Cref{thm:general-KoksmaHlawka}. Therefore, a more refined approach is necessary to establish our main result. 
This will be the content of the second part of the paper. In \Cref{part: modular}, we first prove two general joint equidistribution results for $n$ functions, each acting individually on Fourier coefficients of a fixed eigenform basis (see \Cref{thm:nforms}). With this result in hand, we are then able to prove in \Cref{sec:transformed} our main theorem of the paper using a delicate approximation argument. More precisely, we approximate multi-dimensional indicator functions of trigonometric level sets by appropriate functions of bounded H-K variations with sufficiently small error terms. The last part of this paper will be about obtaining similar results for distributions of Frobenius traces of elliptic curves modulo $p$. We in particular consider two families of interest and outline how the proofs will work over those families.

\tocless\subsection*{\textbf{Acknowledgements.}}

We are grateful to Andrew Granville, Valeriya Kovaleva, Yuk-Kam Lau, and Yingnan Wang for their valuable comments and helpful discussions.


\part{Higher Dimensional Effective $\mu-$Equidistribution}\label{part:effective-equiditribution}
We will start with some basic facts about multivariable Fourier analysis and uniform distribution modulo $1$. For further definitions and properties on equidistribution in any compact space, we refer the reader to \cite[Chapter 3]{MR419394}.

In what follows, we fix the standard notations $e(x)=e^{2\pi ix}$ and $C(S)=\{f: S\to \R : \ f \ \text{is continuous on} \ S\}$.

For $\mathbf{x}=(x_1,\dots,x_n)\in\R^n$, let $f(\mathbf{x})\in L^1(\R^n)$ be an integrable function. Its $n$-dimensional Fourier transform is defined as 
\[\widehat{f}(m_1,\dots,m_n)=\int_{\R^n}f(x_1,\dots,x_n)e(m_1x_1+\dots+m_nx_n)\,dx_1\dots dx_n.\]
If $f$ is further periodic of period $1$ (i.e. periodic of period $1$ in each variable), then $f$ can be given as a Fourier series
\[f(x_1,\dots,x_n)=\sum_{(m_1,\dots,m_n)\in\Z^n}\widehat{f}(m_1,\dots,m_n)e(m_1x_1+\dots+m_nx_n).\]

An $n$-dimensional sequence $\left\{(x_1(i_1),\dots,x_n(i_n))\right\}_{i_1, \ldots, i_n}\in[0,1]^n$ is said to be uniformly distributed in $[0,1]^n$ 
if, for every box $I = I_1 \times \cdots \times I_n \subset [0,1]^n$ (i.e., a product of closed intervals), one has
\[\lim_{X_1,\dots,X_n\to\infty}\frac{\# \{ i_1\leq X_1,\dots, i_n\leq X_n:(x_1(i_1),\dots,x_n(i_n))\in I \}}{X_1\cdots X_n}=\lambda(I),\]
where $\lambda$ is the Lebesgue measure on $\R^n$. One can show using Weyl's criterion \cite{Weyl} and the fact that the set of trigonometric polynomials in $\R^n$ is dense in $C([0,1]^n)$, that  $\left\{(x_1(i_1),\dots,x_n(i_n))\right\}_{i_1, \ldots, i_n}$ is uniformly distributed in $[0,1]^n$ if and only if  for every such box $I \subset [0,1]^n$ and every continuous function $f$ on $[0,1]^n$, we have 
\[\lim_{X_1,\dots,X_n\to\infty}\sum_{i_1, \ldots, i_n}f(x_1(i_1),\dots,x_n(i_n))=\int_{I}f(y_1,\dots,y_n)\,dy_1\dots dy_n,\]
where $\sum_{i_1, \ldots, i_n}$  will denote $\sum_{\substack{i_1\leq X_1\\ \vdots\\i_n\leq X_n}}$ here and throughout \Cref{part:effective-equiditribution}.

In this light, we define the following (cf. \cite[Theorem 6.1 of Chapter 1 and Definition 1.1 of Chapter 3]{MR419394}).

\begin{definition}
    Let $\left\{(x_1(i_1),\dots,x_n(i_n))\right\}_{i_1, \ldots, i_n}$ be an $n$-dimensional sequence in $[0,1]^n$, and let $\mu$ be a probability measure on $[0,1]^n$. Then the sequence  is said to be equidistributed in $[0,1]^n$ with respect to $\mu$ if for any function $f\in C([0,1]^n)$, one has 
    \[\lim_{X_1,\dots,X_n\to\infty}\sum_{i_1, \ldots, i_n}f((x_1(i_1),\dots,x_n(i_n))=\int_{[0, 1]^n}f(y_1,\dots,y_n)\,\mu(y_1,\dots,y_n).\]

\end{definition}

\section{A Multi-dimensional Erd\"{o}s-Tur\'{a}n Inequality for Arbitrary Measures}\label{sec:equidistribution-method}
For $z\in \C$, let $K(z)$ be the F\'{e}jer kernel defined as
\[K(z)=\left(\frac{\sin(\pi z)}{\pi z}\right)^2,\]
and $K(0)=1$. Denote by $H(z)$ the entire function
\[H(z)=z^2K(z)\left(\sum_{n\in \Z}\frac{\text{sgn}(n)}{(z-n)^2}+\frac{2}{z}\right),\]
where
$\text{sgn}(n)=1 \ \text{if}\ n>0, \ 
-1 \  \text{if}\ n<0 \ \text{and} \ 
0 \  \text{if}\ n=0,$
and $H(0)=0$. For an interval $I=[a,b]\subset [0,1]$, set
\[V_I(z)=\frac{1}{2}\left(H(z-a)+H(b-z)\right) \quad 
\text{and} \quad
W_I(z)=\frac{1}{2}\left(K(z-a)+K(b-z)\right).\]
Now for $z_1,\dots,z_n\in \C$ and intervals $I_1,\dots I_n\subseteq  [0,1]$, define the two functions
\[L^+_{I_1,\dots,I_n}(z_1,\dots,z_n):=\prod_{j=1}^n\left(V_{I_j}(z_j)+W_{I_j}(z_j)\right)\]
and 
\[L^-_{I_1,\dots,I_n}(z_1,\dots,z_n):=\prod_{j=1}^nV_{I_j}(z_j)-\prod_{j=1}^n\left(V_{I_j}(z_j)+2W_{I_j}(z_j)\right)+\prod_{j=1}^n\left(V_{I_j}(z_j)+W_{I_j}(z_j)\right).\]

 Using the above, Cochrane \cite{Cochrane} introduced $n$-dimensional versions of the Beurling-Selberg polynomials. Namely, for any $\mathbf{M}=(M_1, \dots,M_n)\in\Z^n_{\geq 0}$ and $\mathbf{x}=(x_1,\dots,x_n)\in\R^n$, these are 
\[F^{\pm}_{I_1,\dots,I_n,\mathbf{M}}(x_1,\dots,x_n):=\sum_{(r_1,\dots, r_n)\in\Z^n}L_{I_1,\dots,I_n}^{\pm}\left((M_1+1)(x_1+r_1),\dots,(M_n+1)(x_n+r_n)\right).\]
These functions are trigonometric polynomials of degrees at most $M_1+\dots+M_n$ and are given by a multivariable Fourier series
\[F^{\pm}_{I_1,\dots,I_n,\mathbf{M}}(x_1,\dots,x_n)=\sum_{(m_1,\dots, m_n)\in\Z^n}\hat{F}^{\pm}_{I_1,\dots,I_n,\mathbf{M}}(m_1,\dots,m_n)e(m_1x_1+\dots+m_nx_n)\]
with $\hat{F}^{\pm}_{I_1,\dots,I_n,\mathbf{M}}(m_1,\dots,m_n)=0$ if $|m_i|>M_i$ for some $1\leq i\leq n$. The main property of these two polynomials is that they approximate characteristic functions inside the unit hypercube of $\R^n$ nicely as (see \cite[Theorem 1]{Cochrane})
\[F^{-}_{I_1,\dots,I_n,\mathbf{M}}(\mathbf{x})\leq \chi_{I_1\times\dots\times I_n}(\mathbf{x})\leq F^{+}_{I_1,\dots,I_n,\mathbf{M}}(\mathbf{x}),\]
where $\chi_{I_1 \times \cdots \times I_n}$ denotes the characteristic function of $I_1 \times \cdots \times I_n$. 
Additionally, the polynomials and their coefficients satisfy the following properties that we collect in \Cref{Cochrane} below (these are \cite[Eqn. 10]{Cochrane} and the equation preceding it). We first introduce notation. For $\lambda$ the Lebesgue measure on $\R$, $m\in \Z$, $\mathbf{M}=(M_1, \dots,M_n)\in\Z^n_{\geq 0}$ and intervals $I_1,\dots, I_n\subseteq [0,1]$, define
\begin{equation}\label{eq: Delta}
 \Delta_{\mathbf{M}}(I_1,\dots,I_n):=\prod_{j=1}^n\left(\lambda(I_j)+\frac{2}{M_j+1} \right)-\prod_{j=1}^n\left(\lambda(I_j)+\frac{1}{M_j+1} \right),\end{equation}
and
\begin{equation}\label{eq: P_m}
P_m(I_j):=\min \left(\frac{1}{\pi|m|},\lambda(I_j),1-\lambda(I_j)\right).\end{equation}

\begin{lemma}\label{Cochrane}
Given the notation above, we have that
\begin{equation}\label{Cochrane 1}
\left\lvert\widehat{F}_{I_1,\dots,I_n,\mathbf{M}}^{\pm}(m_1,\dots,m_n)\right\rvert\leq \Delta_{\mathbf{M}}(I_1,\dots,I_n)+\prod_{j=1}^nP_{m_j}(I_j),
\end{equation}
and
\begin{equation}\label{Cochrane 2}
\int_{[0,1]^n}\left|F_{I_1,\dots,I_n,\mathbf{M}}^{\pm}(\mathbf{x})-\chi_{I_1}(x_1)\dots\chi_{I_n}(x_n)\right|\,dx_1\dots dx_n\leq \Delta_{\mathbf{M}}(I_1,\dots,I_n).
\end{equation}
\end{lemma}

Now for an $n$-dimensional sequence $(x_1(i_1),\dots,x_n(i_n))\in[0,1]^n$ and intervals $I_1,\dots,I_n\subseteq [0,1]$, define the counting function
\begin{equation}\label{eq: Counting}
 N_{I_1,\dots,I_n}(X_1,\dots,X_n):=\#\{i_1\leq X_1,  i_2\leq X_2,\dots, i_n\leq X_n: \ (x_{1}(i_1),\dots,x_{n}(i_n))\in I_1\times \dots \times I_n \}.
 \end{equation}
 Let $\mu$ be a measure on $[0,1]^n$ given by
$$\mu=G(-x_1,\dots,-x_n)\,dx_1\dots\,dx_n$$
with $G$ an $n$-dimensional function having a convergent Fourier series
$$G(x_1,\dots,x_n)=\sum_{(m_1,\dots,m_n)\in \Z^n}c_{m_1,\dots, m_n}e(m_1x_1+\dots+m_nx_n).$$
Note that the $c_{m_1,\dots, m_n}$ are defined to be the Fourier coefficients of the  function $G(x_1,\dots,x_n)$. Since $\mu([0,1]^n)=1$, we have that $c_{0,0,\dots,0}=1$. Denote by 
\begin{equation}\label{eq:mu-height}
||\mu||:=\sup_{(x_1,\dots,x_n)\in[0,1]^n}|G(x_1,\dots,x_n)|
\end{equation}
and $\mathbf{m}=(m_1,\dots,m_n)$. We prove the following higher dimensional $\mu$-analogue of the Erd\"{o}s-Tur\'{a}n inequality.
\begin{theorem}\label{thm:discrepancy}
Fix the notation as above. For any $M_1,\dots,M_n\in \Z_{\geq 1}$, we have that
    \begin{align*}
     &\left\lvert\frac{1}{X_1\dots X_n}N_{I_1,\dots,I_n}(X_1,\dots,X_n)-\mu(I_1\times\dots\times I_n)\right\rvert \leq \Delta_{\mathbf{M}}(I_1,\dots,I_n)||\mu||\\
     &+\sum_{\substack{\mathbf{m}\neq (0,0,\dots,0)\\|m_1|\leq M_1\\\vdots\\ |m_n|\leq M_n}}\left(\Delta_{\mathbf{M}}(I_1,\dots,I_n)+\prod_{j=1}^nP_{m_j}(I_j)\right) \left\lvert\frac{1}{X_1\dots X_n}\sum_{i_1, \ldots, i_n}e\left(m_1x_1(i_1)+\dots +m_nx_n(i_n)\right)-c_{m_1,\dots ,m_n}\right\rvert.
     \end{align*}
\end{theorem}

The above theorem is very general as we didn't impose any condition on the coefficients $c_{m_1,\dots,m_n}$ that define the measure. In fact, this works for any  measure on $[0,1]^n$ as long as its defining function $G(x_1,\dots,x_n)$ can be represented as its multi-dimensional Fourier series.

\Cref{thm:discrepancy} recovers Cochrane's result \cite[Theorem 2]{Cochrane} when $\mu=\lambda$, the Lebesgue measure on $\R^n$. In the special case $n=1$, \Cref{thm:discrepancy} recovers (with a slight improvement)  \cite[Theorem 8]{MurtySinha2009}, and generalizes and improves \cite[Proposition 7.1]{LauLiWang}. For other forms of Erd\"{o}s-Tur\'{a}n$-$type inequalities on general spaces (e.g. on manifolds or Lie groups) we refer the reader to \cite{MR2728598}, \cite{FuLauXi2024}, and the references therein.  

 Throughout the paper, we will relax notation by interchanging between $F^\pm_{I_1,\dots,I_n,\mathbf{M}}=F^\pm_{I_1,\dots,I_n,M_1,\dots, M_n}$ and also $\Delta_{\mathbf{M}}=\Delta_{M_1,\dots,M_n}$.
\begin{proof}
We write the proof for the case $n=2$, the cases $n>2$ are similar. Let $\{(x_i,y_j)\}_{\substack{i\leq X\\ j\leq Y}}\subset [0,1]^2$ be a two dimensional sequence. We start by noting that
\begin{align}\label{bound1}
  &\sum_{\substack{i\leq X \\ j\leq Y}}\chi_I(x_i)\chi_J(y_j) -XY\left\lvert\sum_{(m_1,m_2)\in \Z^2}\widehat{F}_{I,J,M_1,M_2}^+(m_1,m_2)c_{m_1, m_2}\right\rvert\nonumber\\
  &\leq\left\lvert \sum_{\substack{ i\leq X \\ j\leq Y}}F^+_{I,J,M_1,M_2}(x_i,y_j)-XY\sum_{(m_1,m_2)\in \Z^2}\widehat{F}_{I,J,M_1,M_2}^+(m_1,m_2)c_{m_1,m_2}\right\rvert\nonumber\\
  &\leq\sum_{(m_1,m_2)\in \Z^2}\left\lvert\widehat{F}_{I,J,M_1,M_2}^+(m_1,m_2)\right\rvert\left\lvert\sum_{\substack{i \leq X \\ j\leq Y}}e(m_1x_i+m_2y_j)-XYc_{m_1,m_2}\right\rvert.
\end{align}
Now, we have
\begin{align*}
   \left\lvert \sum_{(m_1,m_2)\in \Z^2}\widehat{F}_{I,J,M_1,M_2}^+(m_1,m_2)c_{m_1,m_2}\right\rvert&=\left\lvert\sum_{(m_1,m_2)\in \Z^2}c_{m_1,m_2}\int_{0}^1\int_{0}^1F_{I,J,M_1,M_2}^+(x,y)e(-m_1x-m_2y)\,dx\,dy\right\rvert\\
    &=\left\lvert\int_{0}^1\int_{0}^1F_{I,J,M_1,M_2}^+(x,y)G(-x,-y)\,dx\,dy\right\rvert\\
    &\leq ||\mu|| \int_{[0,1]^2}\left\lvert F_{I,J,M_1,M_2}^+(x,y)-\chi_{I}(x)\chi_J(y) \right\rvert\,dx\,dy+\mu(I\times J)\\
    &\leq \Delta_{M_1, M_2} ||\mu||+\mu(I\times J),
\end{align*}
where the last inequality follows from \eqref{Cochrane 2} of Lemma \ref{Cochrane}. This combined with \eqref{bound1} will give
\begin{align*}
&\sum_{\substack{i\leq X \\ j\leq Y}}\chi_I(x_i)\chi_J(y_j) -XY\mu(I\times J)\\
&\leq XY\Delta_{M_1, M_2}||\mu||+\sum_{(m_1,m_2)\in \Z^2}\left\lvert\widehat{F}_{I,J,M_1,M_2}^+(m_1,m_2)\right\rvert\left\lvert\sum_{\substack{ i\leq X \\  j\leq Y}}e(m_1x_i+m_2y_j)-XYc_{m_1,m_2}\right\rvert.
\end{align*}
Bounding $\left\lvert\widehat{F}_{I,J,M_1,M_2}^+(m_1,m_2)\right\rvert$ using \eqref{Cochrane 1} of Lemma \ref{Cochrane}, we get the desired upper bound.
For the reverse inequality, we write 
\begin{align*}
&\sum_{\substack{ i\leq X \\ j\leq Y}}\chi_I(x_i)\chi_J(y_j)-XY\mu(I\times J)\\
&=\sum_{\substack{ 1\leq i\leq X \\  1\leq j\leq Y}}\chi_I(x_i)\chi_J(y_j)+XY\int_{[0,1]^2}\left(F_{I, J, M_1, M_2}^-(x,y)-\chi_I(x)\chi_J
(y)\right)\,\mu-XY\int_{[0,1]^2}F_{I, J, M_1, M_2}^-(x,y)\,\mu.
\end{align*}
Using the bound \eqref{Cochrane 2} of the Lemma \ref{Cochrane}, this is
\begin{align*}
&\geq \sum_{\substack{ i\leq X \\   j\leq Y}} F_{I, J, M_1, M_2}^-(x_i,y_j)-XY\Delta_{M_1, M_2}|\mu||-XY \int_{[0,1]^2}\sum_{(m_1,m_2)\in \Z^2} F_{I, J, M_1, M_2}^-(x,y)c_{m_1,m_2}e(-m_1x-m_2y)\,dx\,dy,
\end{align*}
which is equal to
\begin{align*}
&-XY\Delta_{M_1, M_2} ||\mu||+\sum_{\substack{ i\leq X \\  j\leq Y}}\sum_{(m_1,m_2)\in \Z^2}\widehat{F}_{I,J,M_1,M_2}^-(m_1,m_2)e(m_1x_i+m_2y_j)-XY\sum_{(m_1,m_2)\in \Z^2}\widehat{F}_{I,J,M_1,M_2}^-(m_1,m_2)c_{m_1,m_2}\\
&=-XY\Delta_{M_1, M_2}||\mu||+\sum_{(m_1,m_2)\in \Z^2}\widehat{F}_{I,J,M_1,M_2}^-(m_1,m_2)\left(\sum_{\substack{ i\leq X \\ j\leq Y}}e(m_1x_i+m_2y_j)-XYc_{m_1,m_2}\right),
\end{align*}
and now the bound \eqref{Cochrane 1} of Lemma \ref{Cochrane} will thus imply the desired lower bound
\begin{align*}
&\geq -XY\Delta_{M_1, M_2}||\mu||-\sum_{(m_1,m_2)\in \Z^2}(\Delta_{M_1,M_2}(I,J)+P_{m_1}(I)P_{m_2}(J))\left\lvert\sum_{\substack{i\leq X \\  j\leq Y}}e(m_1x_i+m_2y_j)-XYc_{m_1,m_2}\right\rvert.
\end{align*}
\end{proof}

\section{The Hardy-Krause Variation and an Efficient Koksma-Hlawka Type Result}

In this section, we will give  a version of the Koksma-Hlawka inequality, stated in  \Cref{thm:general-KoksmaHlawka}, which we will apply in later sections. Roughly speaking,  the classical Koksma-Hlawka inequality gives an upper bound for the distance between the Lebesgue integral of a multi-dimensional  function and the average of the
function over a sequence of points, in terms of the variation of the function and the discrepancy of the sequence.  In \cite[Theorem 2.3 and Corollary 2.4]{MR1914223}, G\"otz generalized the inequality from the Lebesgue measure setting to arbitrary measures (see also \cite[Theorem 1]{AiDi2015}).  Later,   Brandolini et al \cite{MR3018136}  generalized the result by 
replacing the integration domain $[0, 1]^n$ with an arbitrary bounded Borel subset of $\R^n$. For more background on this subject, we refer the reader to \cite{Breneis2020}.

We first introduce some notation following \cite{MR1914223}. 
Let $P$ be a partition of $[0, 1]^{n}$:
\[
0=y_j^{(1)}<y_j^{(2)}\ldots <y_j^{(N_j)}=1, \quad  1\leq j\leq n 
\]
 Let $f$ be a function on $[0, 1]^n$. For each $1\leq j\leq n$, define the $j$-th difference operator $\Delta_{j, P}$ on $f$  by 
\begin{align*}
& \Delta_{j,P} (f)(x_1, \ldots, x_{j-1}, y_j^{(i)}, x_{j+1}, \ldots, x_n) :=\\
&  f(x_1, \ldots, x_{j-1}, y_j^{(i+1)}, x_{j+1}, \ldots, x_n)-f(x_1, \ldots, x_{j-1}, y_j^{(i)}, x_{j+1}, \ldots, x_n), \quad 1\leq i\leq N_j-1.
\end{align*}
For $1\leq \ell\leq n$ and pairwise distinct indices $1\leq j_1, \ldots, j_\ell\leq n$, denote by
\[
\Delta_{j_1, \ldots, j_\ell, P}:=\Delta_{j_1, P}\circ\ldots \circ \Delta_{j_\ell, P}
\]
The Vitali variation of $f$ is defined by 
\begin{equation}\label{eq:variation}
V^{(n)}(f):=\sup_{\substack{P \text{ is a partition}\\ \text{ of $[0, 1]^n$}}}\sum_{1\leq i_1\leq N_1-1}\cdots \sum_{1\leq i_n\leq N_n-1}|\Delta_{1, \ldots, n, P} (f)|
\end{equation}
This invariant measures the total variation of the oscillation of $f$ in a multi-dimensional setting. 
However, it does not ``see"  fluctuations of the function on its lower-dimensional faces. (see  e.g., \cite[Example 3.1.12]{Breneis2020}). This motivates the introduction of a finer notion of variation.

For a non-empty set $M\subseteq \{1, \ldots, n\}$ denote by $f^{(M)}$ the restriction of $f$  by setting $x_j=1$ for all $j\notin M$. If for every set $M$, $V^{(\#M)}(f^{(M)})$ is bounded, then $f$ is said to be of bounded variation in the sense of Hardy and Krause, and the quantity 
\begin{equation}\label{eq:HK-variation}
    V^*(f):=\sum_{{\emptyset\neq M\subseteq \{1, \ldots, n\}}} V^{(\#M)}(f^{(M)})
\end{equation}
is called the \textit{Hardy-Krause (H-K) variation} of $f$. The H-K variation is a suitable notion to capture the variations of $f$ restricted to the faces. 

We record the following properties of  functions of bounded H-K variation.  For a function $f:[0, 1]^n\to \mathbb{R}$, we denote by $||f||_\infty :=\sup_{(x_1, \ldots, x_n)\in [0, 1]^n} \left|f(x_1, \ldots, x_n)\right|$. 

\begin{proposition}\label{prop:HK-variation-prop}
    Let $f$ and $g$ be two real valued functions. Let $n, m\geq 1$. 
    \begin{enumerate}
        \item 
   The function $f$ is of bounded H-K variation if and only if it can written as the difference of two completely monotonic functions.\footnote{If $n=1$, completely monotonic is equivalent to monotonic. For $n\geq 2$, see \cite[Definition 2]{MR1631844} for its definition. }
    \item If $f, g: [0, 1]^n\to \R$ are of bounded H-K variation, then for any $\lambda_1, \lambda_2\in \mathbb{R}$, we have  
    \[
    V^*(\lambda_1f+\lambda_2g)\leq |\lambda_1| V^*(f)+|\lambda_2|V^*(g). 
    \]
    \item Let $\mathcal{P}(f, g):=f(x_1, \ldots, x_n) g(y_1, \ldots, y_m)$ be a function defined on $[0, 1]^{n+m}$. 
    If  $f$ and $g$ are of bounded H-K variation, then \[
 V^*(\mathcal{P}(f, g))\ll_{n, m} ||f||_\infty V^*(g)+||g||_\infty V^*(f).
    \]
    \end{enumerate}
\end{proposition}
\begin{proof}
    (1) is   \cite[Theorem 3]{MR1631844}, (2) is \cite[Proposition 3.3.1]{Breneis2020}. 
    
    For part (3), using part  (2) and \cite[Proposition 3.9.6]{Breneis2020}, we obtain  
     \[
 V^*(\mathcal{P}(f, g))\ll_{n, m} ||f||_\infty \max_{\emptyset\neq M'\subseteq \{1, \ldots, m\}}\left( V^{(\#M')}(g^{(M')})\right)+||g||_\infty \max_{\emptyset\neq M\subseteq \{1, \ldots, n\}}\left(V^{(\# M)}(f^{(M)})\right).  
    \]
    Since $V^*(f)>\max_{\emptyset\neq M\subseteq \{1, \ldots, n\}}\left( V^{(\#M)}(f^{(M)})\right)$ and a similar inequality also holds for $g$,  the desired bound follows. 
\end{proof}

The following lemma will be used frequently in our work in  \Cref{part: modular} and \Cref{part: elliptic} of the  paper. 
\begin{lemma}\label{lem:variation}
Let $F_j: [-1,1]\to \R$ and $J_j=[a_j, b_j]$ be a closed interval for all $1\leq j\leq n$, and denote by 
    $
    \alpha_{j}(y):=\#\{x\in [-1, 1]: F_j(x) =y \}.
    $
  Let 
  \begin{equation*}
       f(\theta_1,\ldots, \theta_n) =\begin{cases}
         \chi_{J_1}(F_1(\cos(2\pi \theta_1)))\cdots \chi_{J_n}(F_n(\cos(2\pi \theta_n)))  & \text{ if $\theta_1, \ldots, \theta_n\in[0, 1/2]$}\\
         0 & \text{ if $\theta_1, \ldots, \theta_n\in (1/2, 1]$}
       \end{cases}.
  \end{equation*}
Then 
$ V^*(f)=O_n\left(\sum_{1\leq i\leq n} \left(\alpha_{i}(a_i)+ \alpha_{i}(b_i)\right)\right)$.
\end{lemma}
\begin{proof}
     By part (3) of \Cref{prop:HK-variation-prop}, it suffices to bound $V^*(\chi_{J_j}(F_j(\cos(2\pi \theta_j))))$ for each $1\leq j\leq n$. We set  $f_j:=\chi_{J_j}(F_j(\cos(2\pi \theta_j)))$ and 
denote by $P$ a partition of $[0, 1]$:
    \[
    0=y^{(1)}<y^{(2)}<\ldots < y^{(N_1)}=1.
    \]
Then, by definition,     
    \begin{align*}
    V^{*}(f_j)=V^{(1)}(f_j)& =\sup_P\sum_{1\leq i_1\leq N_1-1}|\Delta_{1, P} (f_j)(y^{(i_1)})|\\
    & =\sup_P\sum_{1\leq i_1\leq N_1-1}|f_j(y^{(i_1+1)})-f_j(y^{(i_1)})|\\
    &=\sup_P\sum_{1\leq i_1\leq N_1-1}|\chi_{J_j}(F_j(\cos(2\pi y^{(i_1+1)}))-\chi_{J_j}(F_j(\cos(2\pi y^{(i_1)}))|\\
    & \leq  \alpha_{j}(a_j)+ \alpha_{j}(b_j),
    \end{align*}
    where  we use the monotonicity of $\cos(2\pi \theta)$ when $\theta\in [0, 1/2]$ in the last step.
The bound will now follow from part (3) of \Cref{prop:HK-variation-prop}.
\end{proof}

 Next, we introduce a result due to G\"otz \cite[Corollary 2.4]{MR1914223} that we will use to prove our version of the Koksma-Hlawka inequality. Let  $\mu$ and $\nu$ be two probability measures defined on the $n$-dimensional Euclidean space $[0, 1]^n$. We define the invariant
\begin{equation}\label{eq:discrepancy}
 D_n^*[\mu, \nu]:=\sup\{|\mu(B)-\nu(B)|: B =  [0, 1]^n\cap [0, X_1)\times \cdots \times [0, X_n), X_j>0 \text{ for all } 1\leq j \leq n\},   
\end{equation}
measuring the distance between $\mu$ and $\nu$.

\begin{theorem}[G\"otz]\label{Gotz}
  Let $f$ be a function defined on  $[0, 1]^n$ with  bounded H-K variation. Let $\mu$ and $\nu$ be arbitrary probability measures on $[0, 1]^n$. Then 
  \[
  \left|\int f \ d\mu- \int f \ d\nu\right|\leq V^*(f)D_n^*[\mu, \nu].
  \]
\end{theorem}

We are now ready to give our version of the Koksma-Hlawka$-$type inequality.  Given an $n$-dimensional sequence of vectors $\{(x_1(i_1),\ldots, x_n(i_n))\}_{\substack{1\leq i_j\leq X_j\\1\leq j\leq n}}$  in $[0, 1]^n$,  which is $\mu$-equidistributed, we define the \textit{$\mu$-discrepancy} of the sequence as
    \begin{equation}\label{eq: defDisc}
    D_n(\mu):=\sup_{I_1\times\cdots \times I_n\subseteq [0, 1]^n}\left|\sum_{i_1, \ldots, i_n}\chi_{I_1\times \cdots \times I_n}(x_1(i_1),\ldots, x_n(i_n)))-\mu(I_1\times \cdots \times I_n)X_1\ldots X_n\right|.
    \end{equation}

 For $\mathbf{m}=(m_1,\dots,m_n)\in \Z^n$ denote by \[\cP_{\mathbf{m}}:=\begin{cases}
\prod_{\substack{i=1}}^n\frac{1}{|m_i|}   & \text{ if $m_1,\dots, m_n\neq 0$}\\
\prod_{\substack{i\neq j_1,\dots,j_r}}\frac{1}{|m_i|}  & \text{ if $m_{j_1},\dots,m_{j_r}=0$ and $m_i\neq 0$ for $i\neq j_1,\dots,j_r$}.
\end{cases}\]

Recall definitions and notation from \Cref{sec:equidistribution-method}, in particular, \eqref{eq: Delta}, \eqref{eq: P_m}, and \eqref{eq:mu-height}.  We will first need the following estimates for the  terms appearing in \Cref{thm:discrepancy}.  
 
\begin{lemma}\label{lem:bound}
For $I_1,\dots,I_n\subset [0,1]$. Then for sufficiently large $m_1,\dots, m_n$, 
    \[
   \sup_{I_1\times \dots\times I_n} \Delta_{m_1,\dots, m_n}(I_1,\dots,I_n)\ll_n \max\left\{\frac{1}{|m_1|+1},\dots, \frac{1}{|m_n|+1} \right\}, 
    \]
    and
    \[
  \sup_{I_1\times \dots\times I_n} |P_{m_1}(I_1)\dots P_{m_n}(I_n)|\ll 
  \cP_{\mathbf{m}}.
    \]
\end{lemma}

\begin{proof} We give the proof for $n=2$, as the other cases follow similarly.

From \eqref{eq: Delta}, upon expanding the products, we see that
    \[ \Delta_{m_1, m_2}(I_1,I_2)
    =\frac{\lambda(I_1)}{m_2+1}+\frac{\lambda(I_2)}{m_1+1}+\frac{3}{(m_1+1)(m_2+1)}.
    \]
    And from \eqref{eq: P_m}, one can bound
\[P_{m_1}(I_1)P_{m_2}(I_2)
\leq \begin{cases}
\min \left(\lambda(I_1),1-\lambda(I_1)\right) \min \left(\frac{1}{\pi|m_2|},\lambda(I_2),1-\lambda(I_2)\right)   & \text{ if $m_1=0$}\\
 \min \left(\lambda(I_2),1-\lambda(I_2)\right) \min \left(\frac{1}{\pi|m_1|},\lambda(I_1),1-\lambda(I_1)\right)    & \text{ if $m_2=0$}\\
  \min \left(\frac{1}{\pi|m_1|},\lambda(I_2),1-\lambda(I_2)\right) \min \left(\frac{1}{\pi|m_1|},\lambda(I_1),1-\lambda(I_1)\right)    & \text{ if $m_1, m_2\neq 0$}.
\end{cases}
\]
 The result then follows by taking supremums over products of  intervals  and letting $m_1, m_2\to \infty$.
\end{proof}


\begin{theorem}\label{thm:general-KoksmaHlawka}
    Let $\{(x_1(i_1),\ldots, x_n(i_n))\}_{\substack{1\leq i_j\leq X_j\\1\leq j\leq n}}$ be a $\mu$-equidistributed sequence in $[0, 1]^n$, where $\mu=G(-x_1,\dots,-x_n)\,dx_1\dots\,dx_n$
with $G$ an $n$-dimensional function  having a convergent Fourier series $$G(x_1,\dots,x_n)=\sum_{(m_1,\dots,m_n)\in \Z^n}c_{m_1,\dots, m_n}e(m_1x_1+\dots+m_nx_n).$$
    
    Then for any function $f$ defined on $[0,1]^n$ of bounded H-K variation, and for any $M_1,\dots,M_n\in \Z_{\geq 1}$, we have 
    \begin{align*}
    &\left|\sum_{\substack{i_1, \ldots, i_n}} f(x_1(i_1),\ldots, x_n(i_n))-X_1\dots X_n\int_{[0, 1]^n} f\ \mu \right|\leq V^*(f)D_n(\mu)\leq  \frac{||\mu||V^*(f)}{1+\min_{1\leq j\leq  n} |M_j|}X_1\dots X_n+\\
     &+V^*(f)\sum_{\substack{\mathbf{m}\neq (0,0,\dots,0)\\|m_1|\leq M_1\\\vdots\\ |m_n|\leq M_n}}\left(\frac{1}{1+\min_{1\leq j \leq  n}|M_j|}+\cP_{\mathbf{m}}\right) \left\lvert\sum_{i_1, \ldots, i_n}e\left(m_1x_1(i_1)+\dots m_nx_n(i_n)\right)-X_1\dots X_nc_{m_1,\dots ,m_n}\right\rvert
    \end{align*}
\end{theorem}
\begin{proof}
 Let $\nu$ be the discrete  measure on $[0, 1]^n$ defined  by 
 \[
 \nu:=\frac{1}{X_1\ldots X_n}\sum_{\substack{i_1, \ldots,  i_n}}\delta_{(x_1(i_1), \ldots, x_n(i_n))},
 \]
 where $\delta_{(x_1(i_1), \ldots, x_n(i_n))}$ is the Dirac measure at the point  $(x_1(i_1), \ldots, x_n(i_n))$. 
    In particular, we have 
    \[
    \nu(I_1\times \cdots \times I_n)=\sum_{\substack{ i_1, \ldots, i_n}}\chi_{I_1\times \cdots \times I_n}(x_1(i_1),\ldots, x_n(i_n))).
    \]
    The result then follows from Theorem \ref{Gotz} by noting that $D_n^*[\mu,\nu]\leq \frac{1}{X_1\cdots X_n} D_n(\mu)$, and then appealing to Theorem \ref{thm:discrepancy} and \Cref{lem:bound} to bound $D_n(\mu)$.
\end{proof}

\part{Modular Forms}\label{part: modular}

\section{The Joint Distribution Case}\label{sec:Joint-distribution}

The main goal of this section is to prove the following theorem. The strategy of the proof is important, and will be used further in \Cref{sec:transformed} and \Cref{sec:FamilyAllCurves}. 

\begin{theorem}\label{thm:nforms}
Let $n$ be a fixed positive integer. For $1\leq j\leq n$ let $I_j=[a_j,b_j]\subset[-1, 1]$, $F_j:\ [-1,1]\to \R$, and  $\alpha_j(y):=\#\{x\in [-1, 1]: F_j(x) =y \}$. Suppose that $F_j^{-1}(I_j)$ is Lebesgue measurable and $\alpha_j(a_j),\alpha_j(b_j)<\infty$ for all $1\leq j\leq n$, and let $C_{F_1,\dots,F_n}=\sum_{j=1}^n\left(\alpha_j(a_j)+\alpha_j(b_j)\right)$.

\begin{enumerate}
    \item[(i)] Fix a prime $p$. For $1\leq j\leq n$, consider $S_{k_j}(N_j)$ of dimension $d_j$ with $p\nmid N_j$. We have that as $d_1, \ldots, d_n\to \infty$, 
    \begin{align}\label{eq:nforms}
  &  \frac{1}{d_1\ldots d_n}\#\left\{1\leq i_1\leq d_1, \ldots, 1\leq i_n\leq d_n: \left(F_1\left(\frac{a_{p, 1}(i_1)}{2p^{\frac{k_1-1}{2}}}\right), \ldots, F_n\left(\frac{a_{p, n}(i_n)}{2p^{\frac{k_2-1}{2}}}\right)\right)\in I_1\times \ldots \times I_n \right\} \\ \nonumber 
    & \hspace{180pt} =\int_{\substack{F_1(x_1)\in I_1, \dots, F_n(x_n)\in I_n\\ (x_1,\dots,x_n)\in [-1,1]^n}} \,\mu_{p, n}+O_n\left(C_{F_1,\dots,F_n} \frac{\log p}{\log (d_1\cdots d_n)} \right),
    \end{align}
    where $\mu_{p, n}$ is the measure given in \eqref{mu_p,n}, and the implied constant depends only on $n$ and can be made explicit. 
    \item[(ii)] Fix $n$ primes $p_1,\dots,p_n$ and let $S_k(N)$ be a space of dimension $d$ with $p_1,\dots,p_n\nmid N$. We have that as $d\to \infty$, 
\begin{align}\label{eq:oneform}
  &  \frac{1}{d}\#\left\{1\leq i\leq d: \left(F_1\left(\frac{a_{p_1}(i)}{2p_1^{\frac{k-1}{2}}}\right), \ldots, F_n\left(\frac{a_{p_n}(i)}{2p_n^{\frac{k-1}{2}}}\right)\right)\in I_1\times \ldots \times I_n \right\} \\ \nonumber 
    & \hspace{60pt} =\int_{\substack{F_1(x_1)\in I_1, \dots, F_n(x_n)\in I_n\\ (x_1,\dots,x_n)\in [-1,1]^n}} \,\mu_{p_1}(x_1)\dots\,\mu_{p_n}(x_n)+O_n\left(C_{F_1,\dots,F_n} \frac{\log (p_1\cdots p_n)}{\log d} \right),
    \end{align}
    where each $\mu_{p_i}$ is a  Plancherel measure given in \eqref{eqn: Plancheral}, and the implied constant depends only on $n$ and can be made explicit. 
\end{enumerate}
\end{theorem}

   Note that the formula in \eqref{eq:nforms} recovers the result of Murty-Sinha \cite[Theorem 2]{MurtySinha2009} for the case $n=1$ and $F_1(x)=x$. In fact, by applying \cite[Theorem 2]{MurtySinha2009} to the disjoint families $S_{k_1}(N_1), \ldots, S_{k_n}(N_n)$, one can easily get a formula as in \eqref{eq:nforms} for the case where $F_j(x)=x$ for all $1\leq j\leq n$, with an error term of size
$
   O\left(\sum_{1\leq j\leq n}\frac{\log p}{\log d_j}\right).
$
However, the error term in \Cref{thm:nforms} is better than the latter if $\min\{d_1, \ldots, d_n\}$ is significantly smaller than $\max\{d_1, \ldots, d_n\}$. As for the formula in \eqref{eq:oneform}, this in particular  recovers a special case of a result of Lau-Wang \cite{LauWang}  where $F_j(x)=x$ for all $1\leq j\leq n$. 

An interesting feature of our proof is that we do not need to know what the equidistribution measure is beforehand. Indeed,  by  \Cref{thm:discrepancy} and \Cref{thm:general-KoksmaHlawka}, we can recover this measure by only studying the error terms. 

\subsection{Proof of \Cref{thm:nforms}}
To make the exposition  clearer, we will start by first proving \eqref{eq:nforms} for the case $n=2$  and  $F_1(x)=F_2(x)=x$ (the case $n=1$ is \cite[Theorem 2]{MurtySinha2009}), and then generalize to any $F_1,F_2$, and then to any $n\geq 2$. At the end we show how proving \eqref{eq:oneform} is similar.



 For $j\in \{1,2\}$, let $\theta_{p,j}(i_j)\in[0,\pi]$ be such that
\[\cos(\theta_{p,j}(i_j))=\frac{a_{p, j}(i_j)}{2p^{\frac{k_j-1}{2}}}.
\]
Set 
\[ \Lambda_{1, 2}:=\sum_{\substack{i_1\leq d_1\\i_2\leq d_2}}\chi_{I_1}(\cos(\theta_{p,1}(i_1)))\chi_{I_2}(\cos(\theta_{p,2}(i_2))).
\]
Then 
 the function on the left hand side of \eqref{eq:nforms} equals 
$\frac{1}{d_1d_2} \Lambda_{1, 2}$.
Define $J_j\subseteq [0,\frac{1}{2}]$ by  
\begin{equation}\label{eq:theta_costheta}
\frac{\theta_{p,j}(i_j)}{2\pi}\in J_j \Longleftrightarrow \cos(\theta_{p,j}(i_j))\in I_j.
\end{equation}
Consider the (extended) sequence $\left\{\left(\pm \frac{\theta_{p,1}(i_1)}{2\pi}\mod 1,\pm \frac{ \theta_{p,2}(i_2)}{2\pi} \mod 1\right)\right\}_{i_1,i_2}\subset [0,1]^2$. We denote by 
\[
\theta_1^{\pm}(i_1):=\pm \frac{\theta_{p,1}(i_1)}{2\pi}\mod 1, \quad  \theta_2^{\pm}(i_2):=\pm \frac{\theta_{p,2}(i_2)}{2\pi}\mod 1. 
\]

Using the definition \eqref{eq: Counting}, we have
\begin{align*}
N_{J_1, J_2}(d_1, d_2) & =
\sum_{\substack{i_1\leq d_1\\ i_2\leq d_2}}\chi_{J_1}(\theta^\pm_{1}(i_1))\chi_{J_2}(\theta^\pm_{2}(i_2)),
\end{align*}
which by \eqref{eq:theta_costheta} equals
\begin{align}\label{eq:chi-Lambda}
4\Lambda_{1, 2}= \sum_{\substack{i_1\leq d_1\\i_2\leq d_2}}\chi_{J_1}(\cos( 2\pi\theta^\pm_1(i_1)))\chi_{J_2}(\cos(2\pi \theta^\pm_2(i_2))).
\end{align}

For $j\in \{1,2\}$, and $m_j\in \Z$, denote by $$c_{m_j}:=\lim_{d_j\to \infty}\frac{1}{2d_j}\sum_{1\leq i_j\leq d_j} 2\cos (m_j\theta_{p,j}(i_j)).$$ By \cite[Theorem 18]{MurtySinha2009}, we have that
\[
c_{m_j}=\begin{cases}
    1 & \text{ if $m_j=0$,} \\
    \frac{1}{2p^{|m_j|/2}}-\frac{1}{2p^{(|m_j|-2)/2}} & \text{ if $|m_j|\geq 2$ is even,}\\
    0 & \text{ if $|m_j|\geq 1$ is odd}.
\end{cases}
\]
Defining $c_{0, 0}:=1$ and  
$c_{m_1, m_2}:=c_{m_1} c_{m_2}$ for all $(m_1, m_2)\neq (0, 0)$, we set
\[
G(x, y):=\sum_{(m_1, m_2)\in\Z^2}c_{m_1, m_2}e(m_1x+m_2y).
\]
Note that
\[
G(x, y)=\left(\sum_{m_1\in\Z}c_{m_1}e(m_1x)\right)\left(\sum_{m_2\in\Z}c_{m_2}e(m_2y)\right).
\]
Hence, from Section 9 of \cite{MurtySinha2009}, it is immediate that 
\[
G_p(x, y)=(p+1)^2\frac{\sin^2(2\pi x)}{(p^{1/2}+p^{-1/2})^2-4\cos^2(2\pi x)}\cdot \frac{\sin^2(2\pi y)}{(p^{1/2}+p^{-1/2})^2-4\cos^2(2\pi y)}.
\]
Using these, we have the following result.

\begin{proposition}\label{prop:measure}
   The sequence $ \left\{\left(\theta_1^{\pm}(i_1),\theta_2^{\pm}(i_2)\right)\right\}_{i_1,i_2}$ is equidistributed in $[0,1]^2$ with respect to the measure $\tilde{\mu}_{p,2}:=G_p(-\theta_1, -\theta_2) \ d\theta_1  d\theta_2$. Furthermore, the sequence $\left\{\left(\cos(\theta_{p,1}(i_1)), \cos(\theta_{p,2}(i_2))\right)\right\}_{i_1,i_2}$ is equidistributed in $[-1,1]^2$ with respect to the measure $
   \mu_{p,2}
   $ defined in \eqref{mu_p,n}. 
\end{proposition}
\begin{proof}
   Observe that the Weyl limits satisfy
   \begin{align*}
 &  \lim_{d_1, d_2\to \infty}\frac{1}{4d_1d_2}\sum_{1\leq i_1\leq d_2}\sum_{1\leq i_2\leq d_2} e\left(m_1\theta_1^{\pm}(i_1)+m_2\theta_2^{\pm}(i_2)\right)\\
   & =\lim_{d_1, d_2\to \infty}\frac{1}{4d_1d_2}\sum_{1\leq i_1\leq d_1}\sum_{1\leq i_2\leq d_2} 2\cos( m_1\theta_{p,1}(i_1)))\cdot 2\cos(m_2\theta_{p,2}(i_2)))\\
   & =\lim_{d_1, d_2\to \infty}\left(\frac{1}{2d_1}\sum_{1\leq i_1\leq d_1}2\cos(m_1\theta_{p,1}(i_1))\right)\cdot \left(\frac{1}{2d_2}\sum_{1\leq i_2\leq d_2}2\cos(m_2\theta_{p,2}(i_2))\right)=c_{m_1, m_2}
   \end{align*}

 Applying \Cref{thm:discrepancy} to the sequence $ \left\{\left(\theta_1^{\pm}(i_1),\theta_2^{\pm}(i_2)\right)\right\}_{i_1,i_2}$, and the measure $\mu=\tilde{\mu}_{p,2}$, for any fixed positive integers $M_1, M_2$, and using the preceding observation, we get, upon taking the limit as $d_1,d_2\to\infty$, that
   \[
  \lim_{d_1, d_2\to \infty}  \left|\frac{1}{4d_1d_2}N_{J_1, J_2}(d_1, d_2)-\tilde{\mu}_{p,2}(J_1\times J_2)\right|\leq \Delta_{M_1, M_2}(J_1, J_2)||\tilde{\mu}_{p,2}||.
   \]
   Now, from the definition of $G(\theta_1, \theta_2)$, we see that $
||\tilde{\mu}_{p,2}||\ll 1$.
   Recalling the definition in \eqref{eq: Delta} and taking the limit as $M_1,M_2\to \infty$,   we get that $\Delta_{M_1, M_2}(I_1, I_2)\to 0$, which proves the first equidistribution result. 
   
   Now, making  the change of variables $\theta_j\mapsto \frac{\cos^{-1}(x_j)}{2\pi}$ and noting that $(\cos (2\pi \theta_{1}^{\pm}(i_1)), \cos(2\pi \theta_{2}^{\pm}(i_2)))=(\cos(\theta_{p,1}(i_1)), \cos(\theta_{p,2}(i_2)))$, we obtain the second equidistribution result.  
\end{proof}

By \eqref{eq:chi-Lambda} and a change of variables, we have 
\begin{equation}\label{eq:N,Lambda}
\left|\frac{1}{4d_1d_2}N_{J_1, J_2}(d_1, d_2)-\tilde{\mu}_{p,2}(J_1\times J_2)\right|=\left|\frac{1}{d_1d_2}\Lambda_{1, 2}-\int_{[0,1]^2}g(\theta_1,\theta_2)\, \tilde{\mu}_{p,2}(\theta_1,\theta_2)\right|,
\end{equation}
where 
\[
g(\theta_1, \theta_2):=\chi_{I_1}(\cos(2\pi \theta_1)) \chi_{I_2}(\cos(2\pi \theta_2)). 
\]
Now, we will use \Cref{thm:general-KoksmaHlawka} to estimate the above error term \eqref{eq:N,Lambda}. 
For this purpose, we first need the following generalization of  \cite[Theorem 18]{MurtySinha2009}.

\begin{lemma}\label{lem:effective_trace}
  Let $\tau(n)$ and $\omega(n)$ denote the number of positive divisors of $n$ and the number of distinct prime divisors of $n$, respectively. For $m_1m_2\neq 0$, we have 
    \begin{align*}
  &  \left|\sum_{\substack{1\leq i_1\leq d_1\\1\leq i_2\leq d_2}}\cos(m_1\theta_{p,1}(i_1))\cos(m_2\theta_{p,2}(i_2))- d_1d_2 c_{m_1, m_2}\right| 
    \ll \\
    &\hspace{4cm}\left(p^{\frac{3m_1}{2}}2^{\omega(N_1)}\log p^{m_1}+\tau(N_1)N_1^{\frac{1}{2}}\right)\cdot \left(p^{\frac{3m_2}{2}}2^{\omega(N_2)}\log p^{m_2}+\tau(N_2)N_2^{\frac{1}{2}}\right)\\
  &\hspace{3cm}+ d_1\cdot \left(p^{\frac{3m_2}{2}}2^{\omega(N_2)}\log p^{m_2}+\tau(N_2)N_2^{\frac{1}{2}}\right) + d_2\cdot \left(p^{\frac{3m_1}{2}}2^{\omega(N_1)}\log p^{m_1}+\tau(N_1)N_1^{\frac{1}{2}}\right).
    \end{align*}

    If $m_2=0$, then we have 
    \[
    \left|\sum_{\substack{1\leq i_1\leq d_1}}\cos(m_1\theta_{p,1}(i_1))- d_1 c_{m_1, 0}\right|  \ll d_2\left(p^{3m_1/2}2^{\omega(N_1)}\log p^{m_1}+\tau(N_1)N_1^{\frac{1}{2}}\right),
    \]
    and a similar inequality holds for $m_1=0$.
\end{lemma}
\begin{proof}
   For each $1\leq j\leq 2$, denote by 
 $E_j:=\sum_{\substack{1\leq i_j\leq d_j}}\cos(m_j\theta_{p, j}(i_j))-c_{m_j}d_j.
$    Then, we have that
    \begin{align*}
  &\sum_{\substack{1\leq i_1\leq d_1}}\cos(m_1\theta_{p, 1}(i_1))\cdot \sum_{\substack{1\leq i_2\leq d_2}}\cos(m_2\theta_{p, 2}(i_2))  - d_1d_2 c_{m_1}c_{m_2}\\
  & =E_1E_2 +c_{m_2}d_2 E_1+c_{m_1}d_1 E_2.
    \end{align*}
  Therefore, it suffices to estimate 
    $
|E_1E_2|+|c_{m_2}d_2E_1|+|c_{m_1}d_1E_2|.
    $
      By \cite[Theorem 18 and Eq (11)]{MurtySinha2009}, we see that 
$     E_j\ll p^{3m_j/2}2^{\omega(N_j)}\log p^{m_j}+\tau(N_j)\sqrt{N_j}.$
   This gives the following bounds for any nonzero even integers $m_1, m_2$, $|c_{m_1}|, |c_{m_2}|\leq 1$: 
    \begin{align*}
    |c_{m_1}d_1E_2| &
    \ll     d_1\cdot \left(p^{3m_2/2}2^{\omega(N_2)}\log p^{m_2}+\tau(N_2)N_2^{\frac{1}{2}}\right), \\
     |c_{m_2}d_2E_1| &
    \ll     d_2\cdot \left(p^{3m_1/2}2^{\omega(N_1)}\log p^{m_1}+\tau(N_1)N_1^{\frac{1}{2}}\right),\\
    |E_1E_2| &
    \ll     \left(p^{3m_1/2}2^{\omega(N_1)}\log p^{m_1}+\tau(N_1)N_1^{\frac{1}{2}}\right)\cdot \left(p^{3m_2/2}2^{\omega(N_2)}\log p^{m_2}+\tau(N_2)N_2^{\frac{1}{2}}\right).
    \end{align*}
    The first statement now follows. If either $m_1=0$ or $m_2=0$, the result is a consequence of \cite[Eq (11)]{MurtySinha2009}. 
\end{proof}

Applying \Cref{thm:general-KoksmaHlawka} with the sequence $\left\{\left(\theta_1^{\pm}(i_1),\theta_2^{\pm}(i_2)\right)\right\}_{i_1,i_2}$, $X_1=2d_1$, $X_2=2d_2$,   $I_1\times I_2$  replaced by $\pm J_1\pmod 1\times \pm J_2\pmod 1$, $\mu=\tilde{\mu}_{p, 2}$,  and with 
$f(\theta_1, \theta_2)=\chi_{I_1}(\cos(2\pi \theta_1))\chi_{I_2}(\cos(2\pi \theta_2))
$, then by \eqref{eq:N,Lambda},
we reach
\[
\left|4 \Lambda_{1, 2}- 4d_1d_2\int_{[0, 1]^2} \chi_{I_1}(\cos(2\pi \theta_1))\chi_{I_2}(\cos(2\pi \theta_2)) \ d \tilde{\mu}_{p, 2}\right|\leq V^*(f) D_{2}(\tilde{\mu}_{p, 2}).
\]

Note that  by change of variables, 
\[
\int_{[0, 1]^2} \chi_{I_1}(\cos(2\pi \theta_1))\chi_{I_2}(\cos(2\pi \theta_2)) \ \tilde{\mu}_{p, 2}=\int_{I_1\times I_2}  \  \mu_{p, 2}.
\]
Since $V^*(f)\ll 1$ by \Cref{lem:variation}, it suffices to bound $D_{2}(\tilde{\mu}_{p, 2})$.

By \Cref{thm:general-KoksmaHlawka},  and \Cref{lem:effective_trace}, we obtain that for any $M_1, M_2\in \Z_{\geq 1}$ and $\epsilon>0$, 
\begin{align*}D_{2}(\tilde{\mu}_{p, 2}) &\ll  d_1d_2  \max\left(\frac{1}{M_1+1}, \frac{1}{M_2+1}\right) +\\&\sum_{\substack{0<|m_1|\leq M_1\\ 0<|m_2|\leq M_2}}\Bigg[\left( \max\left(\frac{1}{M_1+1}, \frac{1}{M_2+1}\right)+\frac{1}{m_1 m_2} \right)\\&\hspace{3cm}\times (M_1p^{3M_1/2+\epsilon}N_1^{\epsilon}+N_1^{1/2+\epsilon})\cdot (M_2p^{3M_2/2+\epsilon}N_2^{\epsilon}+N_2^{1/2+\epsilon})\Bigg]\\&+\sum_{\substack{0<|m_1|\leq M_1\\ 0<|m_2|\leq M_2}} \Bigg[\left( \max\left(\frac{1}{M_1+1}, \frac{1}{M_2+1}\right)+\frac{1}{m_1m_2}\right)\\&\hspace{3cm}\times \left(d_2\left(M_1p^{3M_1/2+\epsilon}N_1^{\epsilon} +N_1^{1/2+\epsilon}\right) +d_1\left(M_2p^{3M_2/2+\epsilon}N_2^{\epsilon} +N_2^{1/2+\epsilon}\right)\right)\Bigg]\\& +\sum_{\substack{0<|m_1|\leq M_1}}\left( \max\left(\frac{1}{M_1+1}, \frac{1}{M_2+1}\right)+\frac{1}{m_1}\right)d_2\left(M_1p^{3M_1/2+\epsilon}N_1^{\epsilon} +N_1^{1/2+\epsilon}\right)\\&+\sum_{\substack{0<|m_2|\leq M_2}}\left(\max\left(\frac{1}{M_1+1}, \frac{1}{M_2+1}\right)+\frac{1}{m_2}\right)d_1\left(M_2p^{3M_2/2+\epsilon}N_2^{\epsilon} +N_2^{1/2+\epsilon}\right),
\end{align*}
where we have used the facts that $\tau(N_i)\ll N_i^{\epsilon}$ and $2^{\omega(N_i)}\ll N_i^\epsilon$ for any $\epsilon>0$. 
Because of the symmetry between $m_1$ and $m_2$, and since $M_1$ and $M_2$ are arbitrary, we may take $M_1=M_2$ and denote by $M$ the common value. Using further $N_i\ll d_i$, we obtain 
\begin{align*}
D_2(\tilde{\mu}_{p, 2})&\ll \frac{d_1d_2}{M}+ M^{3} p^{3M+\epsilon}d_1^{\epsilon}d_2^{\epsilon}+M(d_1d_2^{1/2+\epsilon}+d_2d_1^{1/2+\epsilon})
+M^{2} p^{3M/2+\epsilon}d_1^\epsilon d_2+M^{2} p^{3M/2+\epsilon}d_2^\epsilon d_1. 
\end{align*}
We now fix $\epsilon>0$ and choose $M\sim c(\epsilon)  \frac{\log(d_1d_2)}{\log p}$ with an appropriate constant $c(\epsilon)$
so that 
\[
\frac{d_1d_2}{M}\asymp M^{3}p^{3M+\epsilon} d_1^{\epsilon}d_2^{\epsilon}.
\]
This implies 
\begin{align}\label{eq:discrepancy-bound}
D_2(\tilde{\mu}_{p, 2})\ll d_1d_2\frac{\log p}{\log(d_1d_2)}.
\end{align}

Plugging this back in the main term, we obtain 
\begin{align*}
\left|4 \Lambda_{1, 2}- 4d_1d_2\int_{[0, 1]^2} \chi_{I_1}(\cos(2\pi \theta_1))\chi_{I_2}(\cos(2\pi \theta_2)) \ d \tilde{\mu}_{p, 2}\right| & \ll d_1d_2\frac{\log p}{\log(d_1d_2)}.
\end{align*}
This completes the proof of \eqref{eq:nforms} when $F_1(x)=F_2(x)=x$.

 \subsubsection{ The general $F_1,F_2$ case}\label{subsec:F1F2} Now in general, for any given $F_1,F_2$, we proceed similarly as our discussion above but now using \Cref{thm:general-KoksmaHlawka} with a different function. Namely, we apply \Cref{thm:general-KoksmaHlawka} with the sequence $\left\{\left(\theta_1^{\pm}(i_1),\theta_2^{\pm}(i_2)\right)\right\}_{i_1,i_2}$, $X_1=2d_1$, $X_2=2d_2$,   $I_1\times I_2$  replaced by $\pm J_1\pmod 1\times \pm J_2\pmod 1$, $\mu=\tilde{\mu}_{p, 2}$,  and the function 
\[
f(\theta_1, \theta_2)=\chi_{I_1}(F_1(\cos(2\pi \theta_1)))\chi_{I_2}(F_2(\cos(2\pi \theta_2))).
\]
This is possible because by \Cref{lem:variation}, we have that 
\begin{align*}
V^*(f)&\leq \#\{x\in [-1, 1]: F_1(x) =a_1 \}+\#\{x\in [-1, 1]: F_1(x) =b_1 \}\\
&+\#\{x\in [-1, 1]: F_2(x) =a_2 \}+\#\{x\in [-1, 1]: F_2(x) =b_2 \},
\end{align*}
and by the hypothesis of \Cref{thm:nforms}, we see that the function $g$ is thus of bounded H-K variation. Setting as before
\[ \Lambda_{F_1, F_2}:=\frac{1}{4}\sum_{\substack{1\leq i_1\leq d_1\\1\leq i_2\leq d_2}}f\left(\theta_1^{\pm}(i_1),\theta_2^{\pm}(i_2)\right), 
\]
we get that
\[
\left|4 \Lambda_{F_1, F_2}- 4d_1d_2\int_{[0, 1]^2} f(\theta_1, \theta_2) \ d \tilde{\mu}_{p, 2}\right|\leq V^*(f) D_{2}(\tilde{\mu}_{p, 2}).
\]
From the work above this is $$\ll C_{F_1,F_2} d_1d_2\frac{\log p}{\log (d_1d_2)}.$$
Finally, using \Cref{prop:measure} will complete the proof of  \eqref{eq:nforms} in the case $n=2$.

\subsubsection{ The $n$-dimensional case} We now show how the proof above will work similarly for any $n$-dimensional case, with suitable adjustments. Recall the notation from the earlier sections and we use $\sum_{i_1, \ldots, i_n}$ to denote $\sum_{\substack{ i_1\leq d_1\\\vdots\\i_n\leq d_n}}$ throughout this section. Let
\[
G_p(x_1,\dots, x_n)=(p+1)^n\frac{\sin^2(2\pi x_1)}{(p^{1/2}+p^{-1/2})^2-4\cos^2(2\pi x_1)}\cdots \frac{\sin^2(2\pi x_n)}{(p^{1/2}+p^{-1/2})^2-4\cos^2(2\pi x_n)},
\]
then one can show, as in \Cref{prop:measure}, that the sequence $ \left\{\left(\theta_1^{\pm}(i_1),\dots, \theta_n^{\pm}(i_n)\right)\right\}_{i_1,\dots,i_n}$ is equidistributed in $[0,1]^n$ with respect to the measure
\begin{equation}\label{eqn:tilde mup,n}
\tilde{\mu}_{p,n}:=\tilde{\mu}_{p,n}(\theta_1, \ldots, \theta_n)=G_p(-\theta_1,\dots, -\theta_n) \ d\theta_1,\dots,  d\theta_n,
\end{equation}
and  the sequence $\left\{\left(\cos(\theta_{p,1}(i_1)),\dots, \cos(\theta_{p,n}(i_n))\right)\right\}_{i_1,\dots,i_n}$ is equidistributed in $[-1,1]^n$ with respect to the measure $
   \mu_{p,n}
   $ defined in \eqref{mu_p,n}. 

Therefore, as above, it suffices to  bound
\[
\left|\frac{1}{d_1\cdots d_n}\Lambda_{n}-\int_{[0,1]^n}g(\theta_1,\dots,\theta_n)\,\tilde{\mu}_{p,n}(\theta_1,\dots,\theta_n)\right|,
\]
where 
\[ \Lambda_{n}:=\sum_{i_1, \ldots, i_n}\chi_{I_1}(\cos(\theta_{p,1}(i_1)))\cdots\chi_{I_n}(\cos(\theta_{p,n}(i_n))),\]
and
\[
g(\theta_1,\dots, \theta_n):=\chi_{I_1}(\cos(2\pi \theta_1))\cdots \chi_{I_n}(\cos(2\pi \theta_n)). 
\]
Again, an application of \Cref{thm:general-KoksmaHlawka} and \Cref{lem:variation} will yield
\begin{equation}\label{eq: 2^n}
\left|\Lambda_{n}- d_1\cdots d_n\int_{I_1\times\cdots \times I_n}  \ \mu_{p, n}\right|\ll_n D_{n}(\tilde{\mu}_{p, n}).
\end{equation}
What remains is to bound $D_{n}(\tilde{\mu}_{p, n})$. To that end, we will need the following generalization of \Cref{lem:effective_trace}.
\begin{lemma}\label{lem:n-case}
Recall the notation in \Cref{lem:effective_trace}. For any $1\leq r\leq n$, if $m_{1},\dots,m_{r}\neq 0$, then
\begin{align*}
  & \left| \sum_{i_1, \ldots, i_r}\cos(m_1\theta_{p,1}(i_1)\dots\cos(m_r\theta_{p,r}(i_r))- d_1\dots d_r c_{m_1,\dots, m_r} \right|
     \\
    &\ll\prod_{j=1}^r \left(p^{\frac{3m_{j}}{2}}2^{\omega(N_{j})}\log p^{m_{j}}+\tau(N_{j})N_{j}^{\frac{1}{2}}\right)+\sum_{k=1}^{r-1}\sum_{\substack{ 1\leq \ell_1<\cdots<\ell_k\leq r}}d_{\ell_1}\dots d_{\ell_k}\prod_{j\neq \ell_1,\dots,\ell_k}\left(p^{\frac{3m_j}{2}}2^{\omega(N_j)}\log p^{m_j}+\tau(N_j)N_j^{\frac{1}{2}}\right).
    \end{align*}
\end{lemma}

\begin{proof}
    Inducting on $n$ and performing the manipulations as in the proof of \Cref{lem:effective_trace}, one obtains
\begin{align*}
&\sum_{\substack{1\leq i_1\leq d_1}}\cos(m_1\theta_{p, 1}(i_1))\dots \sum_{\substack{1\leq i_r\leq d_r}}\cos(m_r\theta_{p, r}(i_r))  - d_1\dots d_r c_{m_1}\dots c_{m_r}\\
&= \prod_{j=1}^r E_j+\sum_{k=1}^{r-1}\,\sum_{\substack{ 1\leq \ell_1<\cdots<\ell_k\leq r}}c_{m_{\ell_1}}\dots c_{m_{\ell_k}}d_{\ell_1}\dots d_{\ell_k}\prod_{j\neq \ell_1,\dots,\ell_k}E_j.
\end{align*}
The proof then follows by bounding each individual term as in the proof of \Cref{lem:effective_trace}. 
\end{proof}
Similarly as in the $n=2$ case, we apply \Cref{thm:general-KoksmaHlawka}, then use \Cref{lem:n-case}, along with taking $M_1=\dots=M_n=M\in\Z_{\geq 1}$ to finally get
\begin{align*}
D_n(\tilde{\mu}_{p, n})&\ll_n \frac{d_1\cdots d_n}{M}+\sum_{r=1}^n \binom{n}{r}\left(M^{r-1}+(\log M)^r\right) \Bigg[\prod_{j=1}^r\left(Mp^{\frac{3M}{2}+\epsilon}d_j^{\epsilon}+d_j^{\frac{1}{2}+\epsilon}\right)\\
&+\sum_{k=1}^{r-1}\sum_{\substack{1\leq \ell_1<\cdots<\ell_k\leq n}}d_{\ell_1}\dots d_{\ell_k}\prod_{j\neq \ell_1,\dots,\ell_k}\left(Mp^{\frac{3M}{2}+\epsilon}d_j^{\epsilon}+d_j^{\frac{1}{2}+\epsilon}\right)\Bigg].
\end{align*}

Fixing an $\epsilon>0$ and choosing $M\sim c(n,\epsilon)  \frac{\log(d_1\cdots d_n)}{\log p}$ with an appropriate constant $c(n, \epsilon)$ so that 
\[
\frac{d_1\cdots d_n}{M}\asymp_n M^{2n-1}p^{\frac{3Mn}{2}+\epsilon} d_1^{\epsilon}\cdots d_n^{\epsilon},
\]
we obtain
\begin{align}\label{eq:n-dim discrepancy-bound}
D_n(\tilde{\mu}_{p, n})\ll_n d_1\cdots d_n\frac{\log p}{\log(d_1\cdots d_n)}.
\end{align}

Finally, the general $F_1,\dots,F_n$ case will follow similarly to what we did for the $F_1,F_2$ case in \Cref{subsec:F1F2}.

\subsubsection{Proof of \eqref{eq:oneform}}\label{subsubsec:LauWangissue} 
The following result was proved in \cite{LauWang} \footnote{Theorem 2 of \cite{LauWang} is in fact written when the level $N=1$, however we were informed by the authors \cite{private} that the theorem is now readily generalizable to any level $N$ using the work in \cite{MurtySinha-newform}. In particular, this follows by replacing Lemma 3.3 of \cite{LauWang} with a version of Proposition 14 of \cite{MurtySinha-newform} for products of prime powers.}.
\begin{align}\label{eq:LauWangoneform}
  &  \frac{1}{d}\#\left\{1\leq i\leq d: \left(\frac{a_{p_1}(i)}{2p_1^{\frac{k-1}{2}}}, \ldots, \frac{a_{p_n}(i)}{2p_n^{\frac{k-1}{2}}}\right)\in I_1\times \ldots \times I_n \right\} \\ \nonumber 
    & \hspace{60pt} =\int_{I_1\times \ldots \times I_n} \,\mu_{p_1}(x_1)\dots\,\mu_{p_n}(x_n)+O_n\left( \frac{\log (p_1\cdots p_n)}{\log d} \right),
    \end{align}
    with an explicit dependence on $n$ in the error. Hence, from \eqref{eq:LauWangoneform}, we have that
\[D_n(\mu_{p_1}\cdots\mu_{p_n})\ll_n d\cdot \frac{\log(p_1\cdots p_n)}{\log d}.\]

Define $\theta_{p_j}\in[0,\pi]$ such that 
\[\cos(\theta_{p_j}(i))=\frac{a_{p_j}(i)}{2p_j^{\frac{k-1}{2}}},
\]
and $J_j$ to be the subset of $[0,\frac{1}{2}]$ with
\begin{equation*}
\frac{\theta_{p_j}(i)}{2\pi}\in J_j \Longleftrightarrow \cos(\theta_{p_j}(i))\in I_j.
\end{equation*}
 Now let  \[
G_{p_1,\dots,p_n}(x_1,\dots, x_n):=\prod_{j=1}^n(p_j+1)\frac{\sin^2(2\pi x_j)}{(p_j^{1/2}+p_j^{-1/2})^2-4\cos^2(2\pi x_j)},
\]
and $\tilde{\mu}_{p_1,\dots,p_n}:=G_{p_1,\dots,p_n}(-\theta_1,\dots, -\theta_n)\,d\theta_1\dots d\theta_n$. 
Then, by a change of variables,  \eqref{eq:LauWangoneform} implies that the sequence $$\left\{\left(\frac{\theta_{p_1}(i)}{2\pi},\dots,\frac{\theta_{p_n}(i)}{2\pi}\right)\right\}_{1\leq i\leq d}$$ is equidistributed in $[0,\frac{1}{2}]^n$ with respect to the measure $\tilde{\mu}_{p_1,\dots,p_n}$ and 
\begin{equation}\label{eq: LauWangdiscrepancy}
    D_n(\tilde{\mu}_{p_1,\dots,p_n})\ll_n d \cdot \frac{\log(p_1\cdots p_n)}{\log d}.
\end{equation}
 
    Therefore, using \eqref{eq: LauWangdiscrepancy}, the proof of $\eqref{eq:oneform}$ will now follow the same lines as we did in \Cref{subsec:F1F2}. Namely, we apply \Cref{thm:general-KoksmaHlawka} to the sequence $\left\{\left(\frac{\theta_{p_1}(i)}{2\pi},\dots,\frac{\theta_{p_n}(i)}{2\pi}\right)\right\}_{1\leq i\leq d}$, the measure $\mu=\tilde{\mu}_{p_1,\dots,p_n}$,  and the function 
\[
f(\theta_1,\dots, \theta_n)=\chi_{I_1}(F_1(\cos(2\pi \theta_1)))\dots \chi_{I_n}(F_n(\cos(2\pi \theta_n))).
\]

\section{Equidistribution for Transformed  Hecke Eigenvalues}\label{sec:transformed}

In this section, we will give the proof of \Cref{thm:main-thm}. 

\subsection{The Setup}\label{sec:setup}


We first introduce the relevant notation of the section. 

Let $J=[a, b]\subset \R$ be a closed interval. Let $n\geq 1$ be an integer and  $F: [-1, 1]^n\to \R$ be any function. Set 
\[
\mathcal{D}_n:=\{(x_1, \ldots, x_n)\in [-1, 1]^n: F(x_1, \ldots, x_n)\in J \}.
\]
For each $1\leq j\leq n$ and $\mathbf{x:=}(x_1, \ldots, x_n)\in [-1, 1]^n$, denote by
$
[\mathbf{x}]_j:=x_j
$
and let
\[
x_{j, \min}:=\inf\{[\mathbf{x}]_j: \mathbf{x} \in \mathcal{D}_n\} \quad \text{and} \quad x_{j, \max}:=\sup\{[\mathbf{x}]_j: \mathbf{x}\in \mathcal{D}_n\}.
\]
Let $N_j$ be any positive integers. 
Consider
\[
x_{j, \min}=x_j(1)<\ldots <x_j(N_j)=x_{j, \max},
\]
a partition of the interval $[x_{j, \min}, x_{j, \max}]$. We denote by 
\[
I_j(i):=[x_j(i), x_j(i+1)], \quad  1\leq i\leq N_j-1.
\]

Now consider
\[
\mathcal{B}_n:=\left\{I_1(i_1) \times \ldots \times I_n(i_n): 1\leq i_j\leq N_j \text{ for all $1\leq j\leq n$}, (I_1(i_1) \times \ldots \times I_n(i_n))\cap \mathcal{D}_n\neq \emptyset \right\}.
\]
Then $\mathcal{B}_n$ forms a cover of $\mathcal{D}_n$. Define the \emph{leftover region}
\[
\mathcal{A}_n :=\bigcup_{J_1\times \ldots \times J_n\in \mathcal{B}_n}(J_1\times \ldots \times J_n)\backslash \mathcal{D}_n,
\]
and it is useful to note that 
\[
\mathcal{A}_n =\bigcup_{\substack{J_1\times \ldots \times J_n\in \mathcal{B}_n\\ J_1\times \ldots \times J_n\nsubseteq \mathcal{D}_n}}\left( (J_1\times \ldots \times J_n)\cap \mathcal{D}_n^c\right),
\]
where $\mathcal{D}_n^c$ is the complement of $\mathcal{D}_n$ in $[-1, 1]^n$. 

We want to adapt the notation above to the setting of \Cref{thm:main-thm}. In this case, our $\mathcal{D}_n$ is Lebesgue measurable, and it has a fixed cover $\cB_n$. Recalling the notation from  \Cref{sec:Joint-distribution}, set 
\begin{equation}
\Lambda_F:=\#\left\{(i_1, \ldots, i_n): 1\leq i_j \leq d_j \text{ for all } 1\leq j\leq n,\; F(\cos(\theta_{p, 1}(i_1)), \ldots, \cos(\theta_{p,j}(i_n)))\in J
\right\}.
\end{equation} 
Denote by 
\[
f(\theta_1, \ldots, \theta_n):=\chi_J(F(\cos(2\pi \theta_1), \ldots, \cos(2\pi \theta_n)))
\]
and consider the extended sequence 
\[
\left\{(\theta_1^\pm(i_1)), \ldots, \theta_n^\pm (i_n)), \ \ \theta_j(i_j)^\pm=\pm\frac{\theta_{p, j}(i_j)}{2\pi} \pmod 1,\; 1\leq j\leq n \right\}.
\]
Then we have that
\begin{equation}\label{eq:rewrite-chracter-sum}
2^n\Lambda_F=\sum_{i_1, \ldots, i_n}f(\theta_1^\pm(i_1), \ldots, \theta_n^\pm(i_n)),
\end{equation}
where we use $\sum_{i_1, \ldots i_n}$ to denote $\sum_{\substack{ i_1\leq d_1\\ \vdots\\  i_n\leq d_n}}$, here and throughout this section. 

We also need the following functions
\[
h(\theta_1, \ldots, \theta_n):=\sum_{J_1\times \ldots \times J_n\in \mathcal{B}_n}\chi_{J_1}(\cos(2\pi \theta_1))\cdots \chi_{J_n}(\cos(2\pi \theta_n))
\]
and
$g(\theta_1, \ldots, \theta_n)=f(\theta_1, \ldots, \theta_n)-h(\theta_1, \ldots, \theta_n).
$
With this notation, observe that 
\begin{equation}\label{eq:function-g}
g(\theta_1, \ldots, \theta_n)=-\chi_{\mathcal{A}_n}(\theta_1, \ldots, \theta_n),
\end{equation}
where $\chi_{\mathcal{A}_n}$ is the characteristic function of $\mathcal{A}_n$. 
The main reason to introduce the function $h$ is that it will be a good approximation of $f$ having bounded H-K variation (see \Cref{subsection:error-term}).  By part (2) of \Cref{prop:HK-variation-prop} and \Cref{lem:variation}, we can get the ``trivial bound":
    \[
    V^*(h)\ll_n \sum_{J_1\times \ldots \times J_n\in \mathcal{B}_n} 2 \ll_n \prod_{1\leq j\leq n} N_j. 
    \]
   However, this estimate is insufficient for our purposes. The following lemma provides a sharper bound. 

\begin{lemma}\label{lem:variation-h}
    We keep the notation as above. If $(F, J)$ is a good pair,  then the H-K variation of $h$ satisfies $V^*(h)=O_{n, F}(\sum_{1\leq j\leq n} N_j)$. 
\end{lemma}
\begin{proof}
   Note that $h$ is also the characteristic function for unions of boxes:
    \[
    h(\theta_1, \ldots, \theta_n):=\chi_{\mathcal{H}}(\cos(2\pi \theta_1),\dots,  \cos(2\pi \theta_n)),
    \]
    where 
    \[
    \mathcal{H}=\bigcup_{J_1\times \ldots \times J_n \in \mathcal{B}_n}J_1\times \ldots \times J_n.
    \]
    By the definition of a good pair, we can find a different covering of $\mathcal{H}$ using a smaller number of boxes. In fact, these boxes can be constructed using the vertices of the boundary of  $ \mathcal{H}$. 
    Therefore, $h$ can be written as a summation of $O_F(\sum_{1\leq j\leq n} N_j)$ many characteristic functions of boxes. 
    The estimation of the variation will now follow from applying \Cref{lem:variation} and part (2) of \Cref{prop:HK-variation-prop}.

    \begin{figure}[h!]
    \centering
    \includegraphics[width=0.35\textwidth]{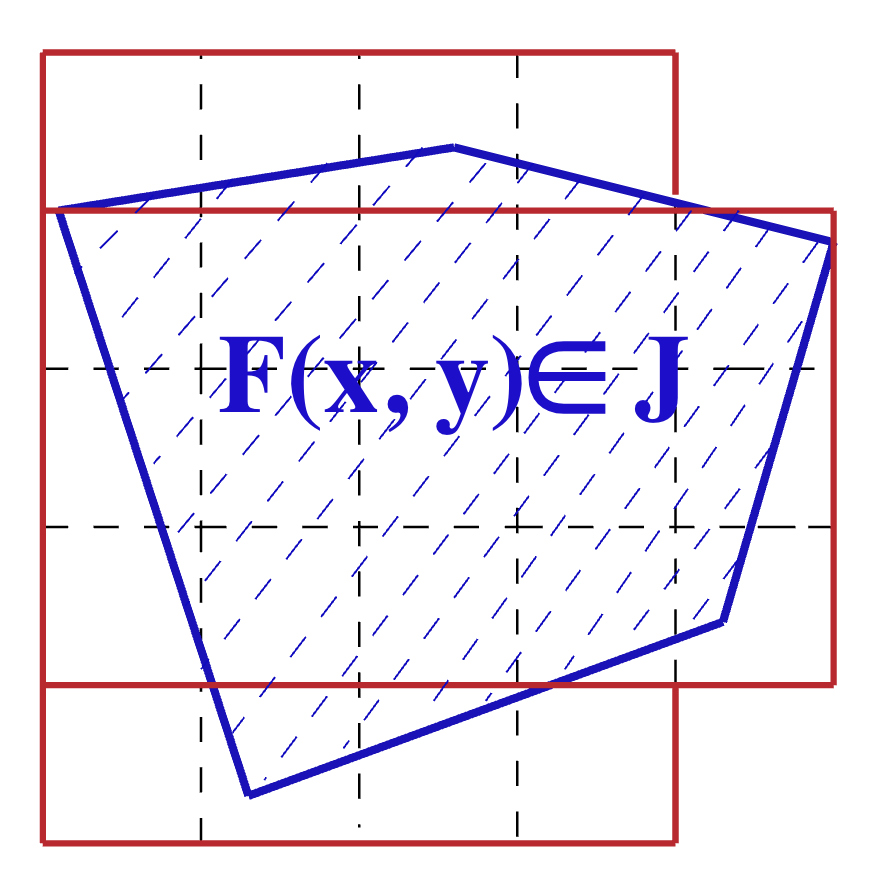}
    \caption{Illustration of \Cref{lem:variation-h} in the two dimensional case. We can use 3 red rectangles to cover $F^{-1}(J)$ rather than 23 smaller squares.}
    \label{fig:graph2}
\end{figure}
\end{proof}

\subsection{Estimating the Error term}\label{subsection:error-term}
 
The goal of this section is to estimate the term  
\begin{equation}\label{eq:error-term}
\left|\sum_{i_1, \ldots, i_n}f(\theta_1^\pm(i_1), \ldots, \theta_n^\pm(i_n))-2^n\prod_{1\leq j\leq n}d_j \int_{[0, 1]^n}f(\theta_1, \ldots, \theta_n) \ \tilde{\mu}_{p, n}\right|,
\end{equation}
where $\tilde{\mu}_{p, n}$ is the equidistribution measure for the sequence 
\begin{equation}\label{eq:extended-seq}
    \left\{\left(\pm \frac{\theta_{p,1}(i_1)}{2\pi}\mod 1,\ldots, \pm \frac{ \theta_{p,n}(i_n)}{2\pi} \mod 1\right)\right\}_{i_1,\ldots, i_n} \subset [0,1]^n
\end{equation} 
given in \eqref{eqn:tilde mup,n}, and we will see later that this term is an error term.  
Note that this is equivalent to estimating 
\[
\left|\Lambda_F-\prod_{1\leq j\leq n}d_j \int_{[0, 1]^n}f(\theta_1, \ldots, \theta_n) \ \tilde{\mu}_{p, n}\right|
\]
 in view of \eqref{eq:rewrite-chracter-sum}.
 
Since $f$ may not be of bounded variation (see \Cref{lem:variation2}), we cannot apply \Cref{thm:general-KoksmaHlawka} directly. The strategy is to split $f$ into two functions, one having bounded variation and the other being small on average.

Applying triangle inequality, we can bound \eqref{eq:error-term} by
\begin{align}
   & \left|\sum_{i_1, \ldots, i_n}h(\theta_1^\pm(i_1), \ldots, \theta_n^\pm(i_n))-2^n\prod_{1\leq j\leq n}d_j \int_{[0, 1]^n}h(\theta_1, \ldots, \theta_n) \ \tilde{\mu}_{p, n}\right| \label{eq:h-error} \\ 
   & +\left|\sum_{i_1, \ldots, i_n}g(\theta_1^\pm(i_1), \ldots, \theta^\pm_n(i_n))\right|  \label{eq:g-sum-error} \\ 
   & +\left|2^n\prod_{1\leq j\leq n}d_j \int_{[0, 1]^n}g(\theta_1, \ldots, \theta_n) \ \tilde{\mu}_{p, n}\right|,  \label{eq:g-integral-error}
\end{align}
where $h$ and $g$ as defined in the setup of \Cref{sec:setup}.
Applying \Cref{thm:general-KoksmaHlawka}, \Cref{lem:variation-h}, and recalling from \eqref{eq:n-dim discrepancy-bound} that 
\[
 D_{n}(\tilde{\mu}_{p, n}) \ll_n \frac{ \prod_{1\leq j\leq n} d_j\cdot \log p}{\log \left(\prod_{1\leq j\leq n} d_j\right)},
\]
we see that \eqref{eq:h-error} is bounded by 
\[
O_{n, F}\left(\frac{\prod_{1\leq j\leq n} d_j}{\log \left(\prod_{1\leq j\leq n} d_j\right)}\cdot  \sum_{1\leq j\leq n} N_j \cdot \log p\right). 
\]

Next, by definition, \eqref{eq:g-sum-error} is equal to 
\begin{align*}
&2^n \sum_{i_1, \ldots, i_n} \chi_{\mathcal{A}_n}(\cos(2\pi \theta_1(i_1)),\cdots,  \cos(2\pi \theta_n(i_n)))
\end{align*}
which is
\[
\leq 2^n \sum_{\substack{J_1\times\ldots \times  J_n \in \mathcal{B}_n \\ J_1\times\ldots \times  J_n \not\subset  \mathcal{D}_n}} \#\left\{(i_1, \ldots, i_n): \cos (2\pi \theta_{j}(i_j))\in J_j, 1\leq i_j \leq d_j \text{ for all $1\leq j\leq n$}  \right\}.  
\]
 Since $(F, J)$ is a good pair, the number of boxes satisfying  $J_1\times\ldots \times  J_n \in \mathcal{B}_n$ and $J_1\times\ldots \times  J_n \not\subset  \mathcal{D}_n$ is bounded by $O_{F}(\sum_{1\leq j\leq n} N_j)$. Thus, the above sum is bounded by  
\begin{align*}
&O_{n, F}\left(\sum_{1\leq j\leq n} N_j \cdot \prod_{1\leq j\leq n} d_j \cdot \max_{\substack{ J_1\times\ldots \times  J_n\in \mathcal{B}_n \\ J_1\times\ldots \times  J_n \not\subset  \mathcal{D}_n}}\left(\int_{J_1\times \ldots \times J_n}  \mu_{p, n}\right)+\sum_{1\leq j\leq n} N_j \cdot \frac{\prod_{1\leq j\leq n} d_j  }{\log \left( \prod_{1\leq j\leq n} d_j\right)} \log p\right)\\
& =O_{n, F}\left( \frac{\sum_{1\leq j\leq n} N_j  }{\prod_j N_j}\cdot \prod_{1\leq j\leq n} d_j+\sum_{1\leq j\leq n} N_j \cdot \frac{\prod_{1\leq j\leq n} d_j  }{\log\left( \prod_{1\leq j\leq n} d_j\right)}\log p \right),
\end{align*}
by applying \Cref{thm:nforms}.

 Finally, we estimate \eqref{eq:g-integral-error}. Using \eqref{eq:function-g}, and the fact that $(F, J)$ is a good pair implies  the  number of boxes in $\cA_n$ is bounded by $O_{ F}\left(\sum_{1\leq j\leq n}N_j\right)$, we reach
\begin{align*}
    2^n  \prod_{1\leq j\leq n}d_j \int_{[0, 1]^n}g(\theta_1, \ldots, \theta_n) \ \tilde{\mu}_{p, n} & =  2^n \prod_{1\leq j\leq n}d_j\cdot  \int_{\mathcal{A}_n} \ \mu_{p, n} \\
    &=2^n \prod_{1\leq j\leq n}d_j \cdot \mu_{p, n}(\mathcal{A}_n)\\
    & \ll_{n, F} \prod_{1\leq j\leq n}d_j \cdot \sum_{1\leq j\leq n} N_j \cdot \max_{\substack{J_1\times\ldots \times  J_n\in \mathcal{B}_n\\J_1\times\ldots \times  J_n \nsubseteq \mathcal{D}_n}}\lambda(J_1\times\ldots \times  J_n)\\
    & \ll_{n, F} \prod_{1\leq j\leq n} d_j \cdot \frac{\sum_{1\leq j\leq n} N_j}{\prod_{1\le j\leq n} N_j},
\end{align*}
where we recall that $\lambda$ is the Lebesgue measure on $\R^n$, and we have used $\mu_{p,n}(J_1\times\dots\times  J_n)\asymp\lambda(J_1\times\dots \times J_n).$

Therefore, we get
\begin{align}\label{eq:parameter}
&\left|\sum_{i_1, \ldots, i_n}f(\theta_1^\pm(i_1), \ldots, \theta_n^\pm(i_n))-2^n \prod_{1\leq j\leq n}d_j \int_{[0, 1]^n}f(\theta_1, \ldots, \theta_n) \ \tilde{\mu}_{p, n}\right| \nonumber \\
&\ll_{n, F} \prod_{1\leq j\leq n}d_j \cdot \frac{\sum_{1\leq j\leq n} N_j}{\prod_{1\le j\leq n} N_j}+  \frac{\prod_{1\leq j\leq n} d_j}{\log \left(\prod_{1\leq j\leq n} d_j\right)}\cdot  \sum_{1\leq j\leq n} N_j \cdot \log p.
\end{align}

Now, choosing
\[
N_1=\ldots=N_n\sim \left\lfloor \left(\frac{\log \left(\prod_{1\leq j\leq n} d_j\right)}{ \log p}\right)^{1/n}\right\rfloor
\]
so that the terms in \eqref{eq:parameter} are asymptotically equal, we get that \eqref{eq:error-term} is  bounded by 

\begin{equation}\label{eq:error-term-bound}
O_{n, F}\left(  d_1\ldots d_n\left(\frac{\log p}{\log \left(\prod_{1\leq j\leq n} d_j\right)}\right)^{1-1/n} \right).
\end{equation}

\subsection{Proof of part (i) of \Cref{thm:main-thm}}\label{subsection:main-term}

Now we proceed to obtain the measure with which the sequence $\left\{ F(\cos(\theta_{p, 1}(i_1)), \ldots, \cos(\theta_{p,n}(i_n)))\right\}_{i_1, \ldots, i_n}$ is equidistributed.

Note that
\begin{align*}
  \int_{[0, 1]^n}f(\theta_1, \ldots, \theta_n) \ \tilde{\mu}_{p, n}& =  \int_{[-1, 1]^n}\chi_J(F(x_1, \ldots, x_n)) \  \mu_{p, n} \nonumber
   = \int_{\substack{F(x_1, \ldots, x_n)\in J\\(x_1, \ldots, x_n)\in [-1, 1]^n}} \ \mu_{p, n},
\end{align*}

Hence by \eqref{eq:rewrite-chracter-sum} and \eqref{eq:error-term-bound},  we have  
\begin{align*}
&\left|\Lambda_F-\prod_{1\leq j\leq n}d_j \int_{\substack{F(x_1, \ldots, x_n)\in J\\(x_1, \ldots, x_n)\in [-1, 1]^n}} \ \mu_{p, n}\right|\\
&=\frac{1}{2^n}
\left|\sum_{i_1, \ldots, i_n}f(\theta_1^\pm(i_1), \ldots, \theta_n^\pm(i_n))-2^n \prod_{1\leq j\leq n}d_j \int_{[0, 1]^n}f(\theta_1, \ldots, \theta_n) \ \tilde{\mu}_{p, n}\right|\\
& \ll_{n, F}  d_1\ldots d_n  \left(\frac{\log p}{\log \left(\prod_{1\leq j\leq n} d_j\right)}\right)^{1-1/n}.
\end{align*}

Finally, for each $n\geq 2$, we improve the error term by adding artificial variables. Namely, define the $n+1$-variable function $G(x_1, \ldots, x_{n+1}):=F(x_1, \ldots, x_{n})$, and choose a space $S_{k_{n+1}}(N_{n+1})$ whose dimension $d_{n+1}$ satisfies   $d_{n+1}\asymp d_1\ldots d_n$. Such a space exists by the dimension formula \eqref{eq:dimension}.

Observe that
\begin{align*}
 \Lambda_F &=\frac{\#\left\{(i_1, \ldots,  i_{n+1}): 1\leq i_j \leq d_j \text{ for all } 1\leq j\leq n+1,\; F(\cos(\theta_{p, 1}(i_1), \ldots,  \cos(\theta_{p,j}(i_{n})))\in J
\right\}}{d_{n+1}}\\
=&\frac{\#\left\{(i_1, \ldots,  i_{n+1}): 1\leq i_j \leq d_j \text{ for all } 1\leq j\leq n+1,\; G(\cos(\theta_{p, 1}(i_1), \ldots,  \cos(\theta_{p,j}(i_{n+1})))\in J
\right\}}{d_{n+1}},
\end{align*}
and 
\begin{align*}
\prod_{1\leq j\leq n+1}d_j \int_{\substack{F(x_1, \ldots, x_n)\in J\\(x_1, \ldots, x_n, x_{n+1})\in [-1, 1]^n}} \ \mu_{p, n+1}
& =d_{n+1}\prod_{1\leq j\leq n}d_j \int_{\substack{F(x_1, \ldots, x_n)\in J\\(x_1, \ldots, x_n)\in [-1, 1]^n}} \ \mu_{p, n}.
\end{align*}
Using what we have proved so far, we have that
\begin{align*}
&d_{n+1}\left|\Lambda_F-\prod_{1\leq j\leq n}d_j \int_{\substack{F(x_1, \ldots, x_n)\in J\\(x_1, \ldots, x_n)\in [-1, 1]^n}} \ \mu_{p, n}\right|\ll_{n, G}
  d_1\ldots d_{n+1}  \left(\frac{\log p}{\log \left(\prod_{1\leq j\leq n+1} d_j\right)}\right)^{1-\frac{1}{n+1}}.
\end{align*}
Note that the dependence of the implicit constant  on $G$ only comes from \Cref{def: goodpair}. If we replace the pair $(F, J)$ by $(G, J)$, then condition (1) of \Cref{def: goodpair} still holds; and by choosing $l_{n+1}=1$ we see that the number of boxes mentioned in condition (2) remains $O_F(\sum_{j=1}^n l_j)$. 
Hence, 
\begin{align*}
&\frac{1}{d_1\ldots d_{n} }\left|\Lambda_F-\prod_{1\leq j\leq n}d_j \int_{\substack{F(x_1, \ldots, x_n)\in J\\(x_1, \ldots, x_n)\in [-1, 1]^n}} \ \mu_{p, n}\right|\ll_{n, F}
   \left(\frac{\log p}{\log \left(\prod_{1\leq j\leq n} d_j\right)}\right)^{1-\frac{1}{n+1}}.
\end{align*}

Now for any $m\geq 1$, one can define the $n+m$-variable function $G(x_1, \ldots, x_{n}, \ldots, x_{n+m}):=F(x_1, \ldots, x_n)$. Similarly,  we can show 
\begin{align*}
&\frac{1}{d_1\ldots d_{n} }\left|\Lambda_F-\prod_{1\leq j\leq n}d_j \int_{\substack{F(x_1, \ldots, x_n)\in J\\(x_1, \ldots, x_n)\in [-1, 1]^n}} \ \mu_{p, n}\right|\ll_{n, m, F}
   \left(\frac{\log p}{\log \left(\prod_{1\leq j\leq n} d_j\right)}\right)^{1-\frac{1}{n+m}}.
\end{align*}
Hence, for any $0<\epsilon\leq 1/n$,  
\begin{align*}
&\frac{1}{d_1\ldots d_{n} }\left|\Lambda_F-\prod_{1\leq j\leq n}d_j \int_{\substack{F(x_1, \ldots, x_n)\in J\\(x_1, \ldots, x_n)\in [-1, 1]^n}} \ \mu_{p, n}\right|\ll_{n, F, \epsilon }
   \left(\frac{\log p}{\log \left(\prod_{1\leq j\leq n} d_j\right)}\right)^{1-\epsilon}.
\end{align*}

\subsection{Proof of \Cref{cor:main-cor}}\label{subsec:CorollaryProof}
Let $\lambda_1, \ldots, \lambda_n\in \R$. If $F(x_1, \ldots, x_n)=\sum_{1\leq j\leq n} \lambda_j x_j$ with $(x_1, \ldots, x_n) \in [-1, 1]^n$, then the region $F^{-1}(J)$
is Lebesgue measurable.
Hence, \Cref{thm:main-thm} is applicable.    
Recalling the definitions of $H_p$ and $h_p$ in \eqref{eqn: Hp} and \eqref{eq: measure hp} respectively, we get by induction that 
\begin{align*}
\int_{\substack{\sum_{1\leq j\leq n} \lambda_j x_j \in J\\(x_1, \ldots, x_n)\in [-1, 1]^n}} \, \prod_{1\leq j\leq n} H_p(x_j)\,d x_1\ldots d x_n &= \int_{\substack{\sum_{1\leq j\leq n} x_j \in J\\x_1\in [-|\lambda_1|, |\lambda_1|] \\\vdots \\ x_n\in [-|\lambda_n|, |\lambda_n|]}} \, \prod_{1\leq j\leq n}  \frac{H_p(x_j/\lambda_j)}{|\lambda_j|}\,d x_1\ldots d x_n \\
& =\int_J h_{p, 1}*\ldots * h_{p, n}(x) \, d x,
\end{align*}
where $*$ is the convolution of functions. 
Hence, the sequence 
\[
\left\{\sum_{1\leq j\le n}\lambda_j\left(\frac{a_{p, j}(i_j)}{2p^{\frac{k_j-1}{2}}}\right)\right\}_{\substack{ i_1\le d_1\\ \vdots \\  i_n\le d_n}}= \left\{\sum_{1\leq j\le n}\lambda_j\cos(2\pi \theta_j(i_j))\right\}_{\substack{ i_1\le d_1\\ \vdots \\   i_n\le d_n}}
\]
is equidistributed with respect to the measure  $\mu_{p, n}^*:=h_{p, 1}*\ldots *h_{p, n}(x) \, d x$ as $d_1, \ldots, d_n\to \infty$. 
The error term follows directly from  \Cref{thm:main-thm}. 

Now we prove the second part of the statement. We note from the proof of \Cref{thm:main-thm} that the error term does not depend on the interval $J$, hence for any $\eta>0$ and $0<\epsilon\leq 1/n$, we have 
\begin{align*}
     &\frac{1}{d_1d_2\dots d_n} \#\left\{1\leq i_1\leq d_1,\dots, 1\leq i_n\leq d_n: \sum_{1\leq j\leq n}\lambda_ja_{p, j}(i_j) = t\right\}\\
     &\leq \frac{1}{d_1d_2\dots d_n} \#\left\{1\leq i_1\leq d_1,\dots, 1\leq i_n\leq d_n: \sum_{1\leq j\leq n}\lambda_j\frac{a_{p, j}(i_j)}{2p^{\frac{k-1}{2}}}\in   \left[\frac{t}{2p^{\frac{k-1}{2}}}-\eta, \frac{t}{2p^{\frac{k-1}{2}}}+\eta \right]\right\}
     \\
     &=\int_{\left[\frac{t}{2p^{\frac{k-1}{2}}}-\eta, \frac{t}{2p^{\frac{k-1}{2}}}+\eta \right]}  \  \mu_{p, n}^*+O_{n,  \epsilon, \lambda_1, \ldots, \lambda_n}\left(  \frac{\log p}{\log \left(d_1d_2\dots d_n\right)}\right)^{1-\epsilon},
\end{align*} 
and the  bound will follow by taking $\eta\to 0$.

\subsection{Proof of part (ii) of \Cref{thm:main-thm} and \Cref{cor:main-cor2}} The proof of part (ii) of \Cref{thm:main-thm} will follow similarly to what we did above for \eqref{eq:extended-seq},  but now working instead with the sequence 
$$\left\{\left(\frac{\theta_{p_1}(i)}{2\pi},\dots,\frac{\theta_{p_n}(i)}{2\pi}\right)\right\}_{1\leq i\leq d}$$ that we saw is equidistributed in $[0,1/2]^n$ with the respect to the measure $$\tilde{\mu}_{p_1,\dots,p_n}:=G_{p_1,\dots,p_n}(-\theta_1,\dots, -\theta_n)\,d\theta_1\dots d\theta_n$$ given in \Cref{subsubsec:LauWangissue}, and satisfying \[D_n(\tilde{\mu}_{p_1,\dots,p_n})\ll_n d\cdot\frac{\log(p_1\cdots p_n)}{\log d}. \]
From this, we get an error term of the form 
\[
O_{n, F}\left(\left(\frac{\log (p_1\cdots p_n)}{\log d} \right)^{1-1/n}\right).
\]
This bound can be improved by using an induction argument as before: adding artificial variables, taking $p_{n+1}=\ldots =p_{n+m}=2$ (or any other prime that doesn't divide the respective levels $N_{k+1},\dots,N_{k+m}$), and inducting on $m$.

The proof of \Cref{cor:main-cor2} will then follow from part (ii) of \Cref{thm:main-thm} likewise to what we did in the previous section.

\part{Elliptic Curves}\label{part: elliptic}
  From the theory of Chebyshev polynomials, one defines for $\ell\in\Z_{\geq 0}$ 
  \begin{equation}\label{eqn:sym}
\sym_{\ell}(\theta):=\frac{\sin((\ell+1)\theta)}{\sin\theta},
\end{equation}
and thus has the formulae
\begin{equation}\label{eqn: cos,sym}
\cos(\theta)=\frac{\sym_1(\theta)}{2} \quad {and}\quad \cos(\ell\theta)=\frac{ \text{sym}_{\ell}(\theta)- \text{sym}_{\ell-2}(\theta)}{2}, \ \ \ell\in\Z_{\geq 2}. 
\end{equation}

 \section{The Family of All Curves}\label{sec:FamilyAllCurves} 
 Recall that for a prime $p\geq 5$, we have the family \[\cF_p=\{E_{a,b}:y^2=x^3+ax+b\;:\; a,b\in\F_p\ \text{and}\ 4a^3\neq-27b^2\} \]
of all elliptic curves over $\F_p$, and that this family has exactly $p(p-1)$ elements. Recall also the works of Birch \cite{Birch68} and Niederreiter \cite{Niederreiter} which establish
\begin{equation}\label{Birch}
\frac{1}{p(p-1)}\#\left\{E\in\cF_p:\ \frac{a_p(E)}{2\sqrt{p}}\in [\alpha,\beta]\subset [-1,1] \right\}= \int_\alpha^\beta\,\mu_{ST}+O(p^{-1/4}).
\end{equation}  
We will first state and prove a joint version of this result.
\begin{theorem}\label{thm:fg}
Let $n$ be a positive integer. For $1\leq j\leq n$, let $I_j=[a_j,b_j]\subset[-1, 1]$, $F_j:\ [-1,1]\to \R$, and  $\alpha_j(y):=\#\{x\in [-1, 1]: F_j(x) =y \}$. Suppose that $F_j^{-1}(I_j)$ is Lebesgue measurable and $\alpha_j(a_j),\alpha_j(b_j)<\infty$ for all $1\leq j\leq n$, and let $C_{F_1,\dots,F_n}=\sum_{j=1}^n\left(\alpha_j(a_j)+\alpha_j(b_j)\right)$. We have that as $p\to \infty$
\begin{align*}
     &\frac{1}{p^n(p-1)^n}\#\left\{E_1, \dots, E_n\in\cF_{p}:\ \left(F_1\left(\frac{a_{p}(E_1)}{2\sqrt{p}}\right),\dots,F_n\left(\frac{a_{p}(E_n)}{2\sqrt{p}}\right)\right)\in I_1\times\dots \times I_n\right\}\nonumber\\
  &\hspace{5cm}=\int_{I_1\times\dots \times I_n}\mu_{ST}(x_1)\dots \mu_{ST}(x_{n})+O_n(C_{F_1,\dots,F_n}\,  p^{-1/4}),
\end{align*}
where the dependence on $n$ in the error can be made explicit.
\end{theorem}

\begin{remark}
Similar to \Cref{thm:nforms}, such a statement, for specific choices of $F_j$, can be proved using the $n=1$ case in \eqref{Birch} and the disjointness of the families. However, we will write a proof following our framework in this paper, which demonstrates that such results can be obtained independent of (knowing) the $n=1$ case, captures the dependency on all $F_j$ in the error, and because we will find the proof useful later.
\end{remark}

To prove  \Cref{thm:fg}, we first start by introducing notation similar to the modular forms case in \Cref{sec:Joint-distribution}. This notation will also be used in the proof of \Cref{thm:ellipticmain-thm} in the next section. For $E_{a,b}\in \cF_p$, let $\theta_p(a,b)=\theta_p(E_{a,b})\in [0,\pi]$ be such that 
\[\cos(\theta_p(a,b))=\frac{a_p(E_{a,b})}{2\sqrt{p}}.\]

Let $J_j\subset[0,1]$ be defined as 
\[\frac{\theta_p(a,b)}{\pi}\in J_j\Longleftrightarrow \cos(\theta_p(a,b))\in I_j,\]
and denote by $$\theta_p^\pm(E_{a,b})=\theta_p^\pm(a,b):=\pm \frac{\theta_p(a,b)}{\pi} \mod 1. $$

\subsection{Proof of \Cref{thm:fg}} 
The proof of \Cref{thm:fg} will be similar to the proof of \Cref{thm:nforms}, so our treatment will be kept brief.

Define our counting function
\[\Lambda_{F_1,\dots,F_n}:=\sum_{\substack{a_1,b_1 \mod p\\ 4a_1^3+27b_1^2\neq 0}}\chi_{I_1}(F_1(\cos(\theta_p(a_1,b_1))))\cdots \sum_{\substack{a_n,b_n \mod p\\ 4a_n^3+27b_n^2\neq 0}}\chi_{I_n}(F_n(\cos(\theta_p(a_n,b_n)))).\]

We will first start by finding the equidistribution measure of the sequence 
\begin{align}\label{theta pm}
    \left\{\theta_p^\pm(E_1),\dots,\theta_p^\pm(E_n) \right\}_{E_1,\dots,E_n\in\cF_p}
\end{align}
as $p\to \infty$.

As in \Cref{sec:Joint-distribution}, we want to use \Cref{thm:discrepancy} and \Cref{thm:general-KoksmaHlawka} to find the right measure for which the sequence in \eqref{theta pm} is equidistributed and estimate the error term. The candidate measure suggested by \Cref{thm:discrepancy}, which we will denote by $\tilde{\mu}_{n}$, is the one satisfying 
\[c_{m_1,\dots,m_n}:=\lim_{p\to\infty}\frac{1}{2^np^n(p-1)^n}\sum_{\substack{a_1,b_1 \mod p\\ 4a_1^3+27b_1^2\neq 0\\ \vdots\\ a_n,b_n \mod p\\ 4a_n^3+27b_n^2\neq 0 }}e\left(m_1\theta_p^\pm(a_1,b_1)+\dots+m_n\theta_p^\pm(a_n,b_n)\right). \]
This splits into a product of separate terms
\[c_{m_1,\dots,m_n}:=\lim_{p\to\infty}\prod_{j=1}^n\frac{1}{2p(p-1)}\sum_{\substack{a,b \mod p\\ 4a^3+27b^2\neq 0}}e\left(m_j\theta_p^\pm(a,b)\right),\]
each of which gives an individual Sato-Tate measure. That is, setting 
\[c_{m_j}=\lim_{p\to\infty}\frac{1}{2p(p-1)}\sum_{\substack{a,b \mod p\\ 4a^3+27b^2\neq 0 }}e\left(m_j\theta_p^\pm(a,b)\right),\]
we have that the normalized Sato-Tate measure (i.e. not $\mu_{ST}$ exactly, but the one corresponding to the normalized angles $\theta_p^{\pm}$) is $G(-x)\,dx$ where 
\[G(x)=2\sin^2(\pi x)=1-\frac{1}{2}\left(e(x)+e(-x)\right),\]
which implies that $c_0=1, c_{\pm 1}=-1/2$, and $c_{m_j}=0$ for all $|m_j|\geq 2.$ Thus
\begin{equation}\label{eqn:cm_1,m_n}
c_{m_1,m_2\dots,m_n}=c_{m_1}c_{m_2}\dots c_{m_n},
\end{equation}
which gives that the candidate measure $\tilde{\mu}_{n}$ satisfies
\[\tilde{\mu}_{n}=G(-x_1)\dots G(-x_n)\,dx_1\dots\,dx_n.\]
To see that this is the right measure, we follow the same logic as in the proof of \Cref{prop:measure}. Namely, applying \Cref{thm:discrepancy} to the sequence in \eqref{theta pm} and then using the facts that
\begin{align*}
   c_{m_1,\dots,m_n}-\prod_{j=1}^n\frac{1}{2p(p-1)}\sum_{\substack{a,b \mod p\\ 4a^3+27b^2\neq 0}}e\left(m_j\theta_p^\pm(a,b)\right) \to 0 \ \ \text{as} \ \ p\to\infty,
\end{align*}
and (recall definition in \eqref{eq: Delta})
\[\Delta_{M_1,\dots,M_n}(J_1,\dots,J_n)\to 0 \ \ \text{as} \ \ M_1,\dots,M_n\to \infty. \]

Now, it suffices to estimate the error terms coming from \Cref{thm:discrepancy} and \Cref{thm:general-KoksmaHlawka} and prove that they are smaller than the main term. To that end, we will need the following result of Katz \cite{Katz88}.

\begin{lemma}[Katz, {Theorem 13.5.3 \cite{Katz88}}]\label{lem:Katz} Let $k\geq 1$ be an integer and
recall the definition in \eqref{eqn:sym}. For a fixed prime $p$, we have that
    \[\frac{1}{(p-1)^2}\sum_{\substack{a,b \mod p\\ 4a^3+27b^2\neq 0 }}\sym_k(\theta_p(a,b))\ll\frac{k}{\sqrt{p}} \]
\end{lemma}
Our coefficients are
\begin{align*}
    c_{m_1,\dots,m_n}=\lim_{p\to\infty}\prod_{j=1}^n\frac{1}{p(p-1)}\sum_{\substack{a,b \mod p\\ 4a^3+27b^2\neq 0}}\cos\left(2m_j\theta_p(a,b)\right),
\end{align*}
and so we want to estimate 
\begin{align*}
   E_p(m_1,\dots,m_n):=\prod_{j=1}^n\sum_{\substack{a,b \mod p\\ 4a^3+27b^2\neq 0}}\cos\left(2m_j\theta_p(a,b)\right)-p^n(p-1)^nc_{m_1,\dots,m_n},
\end{align*}
in view of \Cref{thm:general-KoksmaHlawka}.

 Note that equation \eqref{eqn: cos,sym} together with \Cref{lem:Katz} give for any $m_j\geq 1$, 
 \begin{align}\label{eqn:bound on cos}
  \left\lvert \frac{1}{p(p-1)}  \sum_{\substack{a,b \mod p\\ 4a^3+27b^2\neq 0 }}\cos\left(m_j\theta_p(a,b)\right)\right\rvert \ll \frac{m_j}{\sqrt{p}}.
 \end{align}
  
\begin{lemma}\label{lem: bound for Efg}
Let $k\geq 1$ be an integer. Suppose $m_{i_1},\dots,m_{i_k}\neq 0$ and $m_i= 0$ for all $i\neq i_1,\dots,i_k$.

Then if $|m_{i_{j_0}}|\geq 2$ for some $ 1\leq j_0\leq k$, we have that
\begin{equation*}
    E_{p}(m_1,\dots,m_n)\ll_k p^{2n-\frac{k}{2}}\prod_{j=1}^km_{i_j}, 
\end{equation*}
and otherwise if $|m_{i_{j}}|=1$ for all $1\leq j\leq k$, we have that
\[ E_{p}(m_1,\dots,m_n)\ll_k p^{2n-\frac{1}{2}}. \]
\end{lemma}

\begin{proof}
    First note that from the equality in \eqref{eqn:cm_1,m_n}, we can get the exact values of $c_{m_1,\dots,m_n}$. For example, for the case $n=2$ we have
\[c_{0,0}=1, \quad c_{\pm 1,\pm 1}=\frac{1}{4}, \quad c_{0,\pm1}=c_{\pm1,0}=-\frac{1}{2}, \quad \text{and} \ \ c_{m_1,m_2}=0 \ \ \text{otherwise}.\]
In general for any $n$, we will have $c_{0,\dots,0}=1$, $c_{m_1,\dots,m_n}=0$ if $|m_{i_0}|\geq 2$ for some $1\leq i_0\leq n$, and $c_{m_1,\dots,m_n}=\frac{(-1)^k}{2^k}$ if all $|m_i|\leq 1$ for all $1\leq i\leq n$ and there are exactly $k$ indices $i_j$, $1\leq j\leq k$, with $m_{i_j}\neq 0$. This observation, together with \eqref{eqn:bound on cos},  directly prove the first inequality of the lemma.

Now if $|m_{i_{j}}|=1$ for all $1\leq j\leq k$, we have that $c_{m_1,\dots,m_n}=
\frac{(-1)^k}{2^k}\neq 0$ and thus
\begin{align*}
  \frac{1}{p^n(p-1)^n} E_{p}(m_1,\dots,m_n)= \frac{1}{p^{k}(p-1)^k}\prod_{j=1}^k\sum_{\substack{a,b \mod p\\ 4a^3+27b^2\neq 0 }}\cos\left(2m_j\theta_p(a,b)\right)-\frac{(-1)^{k}}{2^k}.
\end{align*}
Using \eqref{eqn: cos,sym} for the case $\ell=2$ for each $i_j$ factor in the product over $j$, we see that
\begin{align*}
  \frac{1}{p^n(p-1)^n} E_{p}(m_1,\dots,m_n)= \frac{1}{p^{k}(p-1)^k}\left(\sum_{\substack{a,b \mod p\\ 4a^3+27b^2\neq 0 }}\frac{1}{2}\sym(\theta_p(a,b))-\frac{p(p-1)}{2}\right)^k-\frac{(-1)^{k}}{2^k}.
\end{align*}
After expanding the $k$-th power, the constant terms $\frac{(-1)^k}{2^k}$ will cancel and the biggest contribution, using \Cref{lem:Katz}, will then come from the term 
\[ \frac{1}{p^{k}(p-1)^k}\cdot k\cdot \left(\frac{p(p-1)}{2}\right)^{k-1}\sum_{\substack{a,b \mod p\\ 4a^3+27b^2\neq 0 }}\frac{1}{2}\sym(\theta_p(a,b))\ll \frac{k}{\sqrt{p}}.
\]
Estimating all the remaining terms using \Cref{lem:Katz}, will give the second inequality of the lemma.
\end{proof}

We are now ready to apply \Cref{thm:general-KoksmaHlawka}.  Applying \Cref{thm:general-KoksmaHlawka} for the sequence in \eqref{theta pm} and the function 
\[g(\theta_1,\dots,\theta_n):=\chi_{I_1}(F_1(\cos(2\pi\theta_1)))\dots\chi_{I_n}(F_n(\cos(2\pi\theta_n))),\]  together with \Cref{lem:variation}, we get that for any $M_1,\dots, M_n\in \Z_{\geq 1}$,
\begin{align*}\label{D_n(mu_f,g)}
&\left|\Lambda_{F_1,\dots,F_n}-p^n(p-1)^n\int_{I_1\times\dots \times I_n}\mu_{ST}(x_1)\dots \mu_{ST}(x_{n})\right| \ll_n V^*(g)
    D_n(\tilde{\mu}_{n})\\
    &\ll_n V^*(g)p^n(p-1)^n\max_{j=1,\dots,n}\left\{\frac{1}{M_j+1} \right\}+V^*(g)\sum_{\substack{\mathbf{m}\neq (0,\dots,0)\\0\leq |m_1|\leq M_1\\\vdots\\0\leq |m_n|\leq M_n}}\left(\max_{j=1,\dots,n}\left\{\frac{1}{M_j+1} \right\}+\cP_{\mathbf{m}}\right)\left\lvert E_{p}(m_1,\dots,m_n)\right\rvert.
\end{align*}

Using this, with \Cref{lem: bound for Efg}, and taking $M_1=\dots=M_n=M\in\Z_{\geq 1}$ because of the symmetry between the $m_j$'s, we arrive at 
\[ D_n(\tilde{\mu}_{n})\ll_n \frac{p^{2n}}{M}+ p^{2n}\sum_{k=1}^n\frac{M^{2k-1}+M^k}{p^{k/2}}+p^{2n-\frac{1}{2}}\left(1+\frac{1}{M}\right) ,\]
where the $k$-th term in the first sum comes from the case when $m_{j_1},\dots,m_{j_k}\neq 0$ and $m_j= 0$ for all $j\neq j_1,\dots,j_k$, with $|m_{j_0}|>2$ for some $1\leq j_0\leq k$, and the last term comes from the case $m_{j_1},\dots,m_{j_k}=\pm 1$ which make $c_{m_1,\dots,m_n}\neq 0$. We optimize by choosing $M\asymp p^{1/4}$, which implies 
\[D_n(\tilde{\mu}_{n})\ll_n p^{2n-1/4}.\]
Finally, \Cref{lem:variation} gives 
\begin{align*}
V^*(g)&\ll_n  \sum_{j=1}^n\left(\alpha_j(a_j)+\alpha_j(b_j)\right),
\end{align*}
and this will complete the proof of \Cref{thm:fg}.

\subsection{Proof of \Cref{thm:ellipticmain-thm} and \Cref{cor:ellipticmain-cor}}\label{Subsection2 of elliptic curves}
The proof here follows the ideas of the proof of \Cref{thm:main-thm}, and so our treatment will be brief.

Start with a good pair $(F,J)$. Denote the counting function by
\begin{equation}
\Lambda_F:=\#\left\{E_1, \ldots, E_n\in \cF_p:\; F(\cos(\theta_{p}(E_1), \ldots, \cos(\theta_{p}(E_n)))\in J
\right\}.
\end{equation} 

For \[G(x)=1-\frac{1}{2}\left(e(x)+e(-x)\right),\]
recall from the proof of \Cref{thm:fg} that the measure 
\[\tilde{\mu}_n:=G(-x_1)\dots G(-x_n)\,dx_1\dots\,dx_n\]
is the measure for which the sequence 
$\left\{\theta_p^\pm(E_1),\dots,\theta_p^\pm(E_n)\right\}_{\substack{E_1,\dots,E_n\in\cF_p}} $
is equidistributed as $p\to \infty$. 

Now recall the notations and definitions given in \Cref{sec:setup}. We first want to estimate
\begin{equation*}
\left|\sum_{\substack{E_1,\dots, E_n\in\cF_p}}f(\theta_p^\pm(E_1), \ldots, \theta_p^\pm(E_n))-2^np^n(p-1)^n \int_{[0, 1]^n}f(\theta_1, \ldots, \theta_n) \ \tilde{\mu}_{n}\right|,
\end{equation*}
as $p\to \infty,$ where
\[
f(\theta_1, \ldots, \theta_n)=\chi_J(F(\cos(2\pi \theta_1), \ldots, \cos(2\pi \theta_n))).
\]
By our definition, this is the same as estimating 
\[
\left|\Lambda_F-p^n(p-1)^n \int_{[0, 1]^n}f(\theta_1, \ldots, \theta_n) \ \tilde{\mu}_{ n}\right|
\]
since
\begin{equation*}
\Lambda_F=\frac{1}{2^n}\sum_{\substack{E_1,\dots, E_n\in\cF_p}}f(\theta_p^\pm(E_1), \ldots, \theta_p^\pm(E_n)).
\end{equation*}
As in \Cref{subsection:error-term}, it suffices to bound each of 
\begin{align}
   & \left|\sum_{\substack{E_1,\dots E_n\in\cF_p}}h(\theta_p^\pm(E_1), \ldots, \theta_n^\pm(E_n))-2^np^n(p-1)^n \int_{[0, 1]^n}h(\theta_1, \ldots, \theta_n) \ \tilde{\mu}_{ n}\right| \label{eq:elliptich-error1},
\end{align}
   \begin{align}
   &\left|\sum_{\substack{E_1,\dots, E_n\in\cF_p}}g(\theta_p^\pm(E_1), \ldots, \theta^\pm_p(E_n))\right|  \label{eq:ellipticg-sum-error1}, 
\end{align}
and
   \begin{align}
   &\left|2^np^n(p-1)^n  \int_{[0, 1]^n}g(\theta_1, \ldots, \theta_n) \ \tilde{\mu}_{p}\right|,  \label{eq:ellipticg-integral-error1}
\end{align}
where $h$ and $g$ are defined similarly as in \Cref{sec:setup}.

From \Cref{thm:fg}, we have that
$D_n(\tilde{\mu}_n)\ll_n p^{2n-1/4}.$
Therefore, \eqref{eq:elliptich-error1} is bounded by 
\[O_{n,F}\left( p^{2n-1/4}\sum_{j=1}^nN_j\right),\]
 using \Cref{thm:general-KoksmaHlawka} and \Cref{lem:variation-h}. Now, similarly as in \Cref{subsection:error-term}, we can bound \eqref{eq:ellipticg-sum-error1} by 
\begin{align*}
&O_{n,F}\left(p^{2n}\sum_{1\leq j\leq n} N_j \max_{\substack{ J_1\times\ldots \times  J_n\in \mathcal{B}_n \\ J_1\times\ldots \times  J_n \not\subset  \mathcal{D}_n}}\left(\int_{J_1\times \ldots \times J_n}  \mu_{ST}(x_1)\dots\,\mu_{ST}(x_n)\right)+\sum_{1\leq j\leq n} N_j \cdot p^{2n-1/4}\right)\\
& =O_{n,F}\left(p^{2n} \frac{\sum_{1\leq j\leq n} N_j  }{\prod_{1\leq j\leq n} N_j}+\sum_{1\leq j\leq n} N_j \cdot p^{2n-1/4} \right)
\end{align*}
where we have used \Cref{thm:fg},
and bound \eqref{eq:ellipticg-integral-error1} by
\[O_{n,F}\left( p^{2n}\cdot \frac{\sum_{1\leq j\leq n} N_j}{\prod_{1\le j\leq n} N_j}\right),\]
where we have used that $\mu_{ST}(J_1)\times\dots\times  \mu_{ST}(J_n)\asymp\lambda(J_1\times\dots \times J_n)$ 
for the latter. Therefore, 
\begin{align*}
&\left|\sum_{\substack{E_1,\dots, E_n\in\cF_p}}f(\theta_p^\pm(E_1), \ldots, \theta_p^\pm(E_n))-2^np^n(p-1)^n \int_{[0, 1]^n}f(\theta_1, \ldots, \theta_n) \ \tilde{\mu}_{n}\right|\\ 
    &\ll_{n,F} p^{2n-1/4}\sum_{j=1}^nN_j+ p^{2n}\frac{\sum_{1\leq j\leq n} N_j}{\prod_{1\le j\leq n} N_j}
    \end{align*}
and choosing 
\[N_1=\dots=N_n\asymp p^{1/4n},\]
will finally give the bound
\begin{align*}
    \left|\sum_{\substack{E_1,\dots, E_n\in\cF_p}}f(\theta_p^\pm(E_1), \ldots, \theta_p^\pm(E_n))-2^np^n(p-1)^n \int_{[0, 1]^n}f(\theta_1, \ldots, \theta_n) \ \tilde{\mu}_{n}\right|\ll_{n,F} p^{\frac{8n^2-n+1}{4n}}.
\end{align*}
Finally, to get a uniform improved error term for any $n$, we use the latter bound for functions $G(x_1,\dots,x_n,x_{n+1},\dots,x_{n+m})=F(x_1,\dots,x_n)$ and
an induction argument on $m$ as at the end of \Cref{subsection:main-term}. Therefore, the proofs of \Cref{thm:ellipticmain-thm} and \Cref{cor:ellipticmain-cor} will then follow similarly as to what we did in \Cref{subsection:main-term} and \Cref{subsec:CorollaryProof}.


\section{One Parameter Families}
Let $E_{A, B}: \ y^2=x^3+A(T)x+B(T)$ be an elliptic curve over $\Q(T)$. Its $j$-invariant is defined as
\[j(T)=1728\frac{4A(T)^3}{\Delta(T)},\]
where $\Delta(T):=4A(T)^3+27B(T)^2$, and we assume it is not constant, i.e. $j(T)\notin \Q$.
Now for any prime number $p$, and $A(T),B(T)\in\Z[T]$, we specialize $T$ to be a $t\in \F_p$, and we consider the family
\[\cG_{p}=\{y^2=x^3+A(t)x+B(t): \ t\in \F_p \ \ \text{and} \ \ \Delta(t)\neq 0\} \]
This family has $p+O_{A,B}(1)$ curves. Michel \cite{Michel} proved that the normalized traces of the elliptic curves in this family are equidistributed in $[-1,1]$ with respect to the Sato-Tate measure. In particular, he showed
\[\frac{1}{p}\#\left\{E_t\in\cG_{p}: \frac{a_p(E)}{2\sqrt{p}}\in [\alpha,\beta]\subset[-1,1] \right\}\sim \int_{[\alpha,\beta]}\,\mu_{ST} \]
as $p\to \infty$. Miller and Murty \cite{MillerMurty} improved this result by obtaining an error term of size $O(p^{-1/4})$.

The following results are analogous to those of \Cref{thm:ellipticmain-thm} and \Cref{cor:ellipticmain-cor} for the family of all curves $\cF_p$.

\begin{theorem}\label{thm:ff-ellipticmain-thm}
Let $p$ be a prime, $n\geq 2$ be a positive integer, and $(F,J)$ a good pair in the sense of \Cref{def: goodpair}. We have that for any $\epsilon>0$, as $p\to\infty$  
 \begin{align*}
     &\frac{1}{p^n} \#\left\{E_1,\dots,E_n\in \cG_p: F\left(\frac{a_{p}(E_1)}{2\sqrt{p}},\dots, \frac{a_{p}(E_n)}{2\sqrt{p}}\right)\in J \right\}\\
     &\hspace{5cm}=\int_{\substack{F(x_1, \ldots, x_n)\in J\\(x_1, \ldots, x_n)\in [-1, 1]^n}}  \ \mu_{ST}(x_1)\dots \mu_{ST}(x_n) +O_{n,F,\epsilon}\left(p^{-\frac{1}{4}+\epsilon}  \right).
\end{align*} 
\end{theorem}

\begin{corollary}\label{cor:ff-ellipticmain-cor}
 Let $p$ be a prime, $n\geq 2$ a positive integer, and  $J\subset \R$ a closed interval.   Let $\lambda_1, \ldots, \lambda_n$ be nonzero real numbers. Then, for any $\epsilon>0$, as $p\to \infty$ 
 \begin{align*}
     &\frac{1}{p^n(p-1)^n} \#\left\{E_1,\dots,E_n\in \cG_p:\sum_{1\leq j\leq n}\lambda_j\left(\frac{a_{p}(E_j)}{2\sqrt{p}}\right)\in J \right\}=\int_{J}  \ \mu_{ST, n}^* +O_{n,\epsilon,\lambda_1,\dots,\lambda_n}\left( p^{-\frac{1}{4}+\epsilon}  \right),
\end{align*} 
where $\mu_{ST, n}^*:=k_{ 1}*\ldots *k_{n}(x) \, d x$ with $k_j$ defined in \eqref{eq: measure kj} and $*$ is the convolution of functions.
Consequently, we obtain for any $t\in \R$, 
\begin{align*}
     &\frac{1}{p^n(p-1)^n} \#\left\{E_1,\dots,E_n\in \cG_p: \sum_{1\leq j\leq n}\lambda_ja_{p}(E_j)= t\right\}\ll_{n,\epsilon,\lambda_1,\dots,\lambda_n}p^{-\frac{1}{4}+\epsilon} .
\end{align*} 
\end{corollary}

The setup here will be similar to that of Section 5.2. For $E_{A,B,t}\in \cG_p$, let $\theta_p(A,B,t)\in[0,\pi]$ be such that 
\[\cos(\theta_p(A,B,t))=\frac{a_p(E_{A,B,t})}{2\sqrt{p}}.\]

Let $J_j\subset[0,1]$ be defined as 
\[\frac{\theta_p(A,B,t)}{\pi}\in J_j\Longleftrightarrow \cos(\theta_p(A,B,t))\in I_j,\]
and introduce the extended sequence 
$\left\{\theta_p^\pm(A,B,t) \right\}_{\substack{t\in\F_p\\\Delta(t)\neq 0}},$
where $$\theta_p^\pm(A,B,t):=\pm \frac{\theta_p(A,B,t)}{\pi} \mod 1. $$
Using this setup, one can prove \Cref{thm:ff-ellipticmain-thm} and \Cref{cor:ff-ellipticmain-cor} exactly in the same way we did for the family $\cF_p$ in the preceding section, by using the following result of Michel \cite{Michel} for the family $\cG_p$ instead of \Cref{lem:Katz} for $\cF_p$. 

\begin{lemma}[Michel, Proposition 1.1 {\cite{Michel}}]\label{lem:Michel}
    Let $\psi_p$ be an additive character on $\F_p$, and denote by $c(\Delta)$ the number of complex zeros of the discriminant $\Delta(z)$. Then we have the following bound

    \[\left\lvert\frac{1}{p}\sum_{\substack{t\mod p\\ \Delta(t)\neq 0}}\sym_k(\theta_{p}(A,B,t))\psi_p(t)\right\rvert\leq (k+1)\frac{c_{\Delta}-\delta_{\psi_p}-1}{\sqrt{p}},\]
    where $\delta_{\psi_p}=0$ if $\psi_p$ is trivial, and equals to $1$ otherwise. Moreover, we may drop the additive character and the restriction that $\Delta(t)\neq 0$ in the sum, which will only affect the constant in the inequality.
\end{lemma}

Since \Cref{lem:Michel} gives the same bounds as those in \Cref{lem:Katz}, the error terms will not change from  the work over $\cF_p$ in the preceding section.


\begin{appendix}\label{appendix}
\section{Function with Unbounded Hardy-Krause Variation}

The following proposition is not used directly in this paper. However, it shows that even the simpler function arising from \Cref{cor:main-cor} fails to have bounded Hardy-Krause variation. Hence, we cannot apply  \Cref{thm:general-KoksmaHlawka} directly.
As we were unable to find a reference for this result, we include a proof for completeness.

\begin{proposition}\label{lem:variation2}
   Let $\lambda_1, \lambda_2$ be nonzero real numbers. Let $J=[a, b]$ be a closed interval.  The function 
 \begin{equation*}
       f(\theta_1, \theta_2) =\begin{cases}
         \chi_{J}(\lambda_1\cos(2\pi \theta_1)+\lambda_2 \cos(2\pi \theta_2)) & \text{ if $\theta_1,  \theta_2\in[0, 1/2]$}\\
         0 & \text{ if $\theta_1,  \theta_2\in (1/2, 1]$}
       \end{cases}
  \end{equation*}
is not of bounded Hardy-Krause variation. 
\end{proposition}


\begin{proof}

For any partition  $P$ of $[0, 1]^2$:
     \[
    0=y_j(1)<y_j(2)<\ldots < y_j(N_j)=1, \quad 1\leq j\leq 2,
    \]
    we have 
\begin{equation}\label{eq:V_2_f}
 V^{(2)}(f)\geq \sum_{1\leq i_1\leq N_1-1} \sum_{1\leq i_2\leq N_2-1} |\chi_J(Y_{i_1+1, i_2+1})+\chi_J(Y_{i_1, i_2})-\chi_J(Y_{i_1+1, i_2})-\chi_J(Y_{i_1, i_2+1})|,
\end{equation}
where 
\[
Y_{i, j}:=\lambda_1\cos(2\pi y_1(i))+\lambda_2\cos(2\pi y_2(j)), \quad 1\leq i\le N_1-1, 1\le j\leq N_2-1.
\]

  First, assume $\lambda_1>0$ and $\lambda_2> 0$. 
Since $\cos (2\pi\theta)$ is strictly decreasing when $\theta\in  [0, 1/2]$,  we have 
\[
Y_{i_1, i_2}\geq Y_{i_1, i_2+1}, Y_{i_1+1, i_2} \geq Y_{i_1+1, i_2+1}.
\]
For any $X>0$,  we will construct a partition $P$ that makes $V^{(2)}(f)\gg X$, and thus $V^*(f)\gg X$.  
  In fact, we will construct $P$ such that
\begin{equation}\label{eq:unbounded}
\chi_J(Y_{i_1, i_2})=1, \quad \chi_J(Y_{i_1+1, i_2+1})=\chi_J(Y_{i_1, i_2+1})=\chi_J(Y_{i_1+1, i_2})=0.
\end{equation}
for $\asymp X$ many pairs $(i_1, i_2)$ with $1\leq i_1\leq N_1$ and  $1\leq i_2\leq N_2$.

Note that by change of variables $(\theta_1, \theta_2)\mapsto (\lambda_1\cos(2\pi \theta_1), \lambda_2\cos(2\pi \theta_2))$,  a partition $P$ of $[0, 1/2]^2$ has a one to one correspondence with a partition $P'$ of $[-\lambda_1, \lambda_1]\times [-\lambda_2, \lambda_2]$:
\[
 \lambda_j=z_j(1)>z_j(2)>\ldots > z_j(N_j)=-\lambda_j, \quad 1\leq j\leq 2,
\]
where $z_j(i_j)=\lambda_j \cos(2\pi y_j(i_j))$ for all $1\leq i_j\leq N_j$ and $j\in \{1, 2\}$. 
Therefore, to show $V^*(f)\gg X$, it suffices to find a partition  $P'$  such that there are $\asymp X$ many pairs $(z_1(i_1), z_2(i_2))$ satisfying 
\begin{equation}\label{eq:unbounded-2}
 z_1(i_1)+  z_2(i_2)\in J, \quad z_1(i_1+1)+  z_2(i_2+1), z_1(i_1)+  z_2(i_2+1), z_1(i_1+1)+  z_2(i_2)\notin J. 
\end{equation}

Assume $|z_j(i_j)-z_j(i_j+1)|\leq \frac{1}{2\lambda_jX}$ for all $1\leq i_j\leq N_j$ and $i\in \{1, 2\}$. 
Consider the region 
\[
\Omega:=\left\{(z_1, z_2): \; \begin{aligned}   z_1+z_2\in J & \\   -\lambda_1 \leq z_1\leq \lambda_1 & \\
 -\lambda_2 \leq z_2\leq \lambda_2 &\end{aligned}
\right\},
\]
and denote by 
\[
I_j(i_j):=[z_j(i_j), z_j(i_j+1)],\quad 1\leq i_j\leq N_j, \ j\in \{1, 2\}. 
\]
Then, defining $I_1(i_1)\times I_2(i_2)$ for all $1\leq i_1\leq N_1$ and  $1\leq i_2\leq N_2$ is equivalent to giving a partition $P'$. We can define $I_1(i_1)\times I_2(i_2)$ inductively. Starting from  the (left) boundary of $\Omega$, we choose the box $I_1(i_1)\times I_2(i_2)$ ($(i_1, i_2)\in \{(N_1, 1),  (N_1, 2), (N_1-1, 1)\}$) such that
\[
(z_1(i_1), z_2(i_2))\in \Omega, \  (z_1(i_1+1), z_2(i_2+1)), (z_1(i_1+1), z_2(i_2)), (z_1(i_1), z_2(i_2+1))\notin \Omega,
\]
then \eqref{eq:unbounded-2} holds for  the vertices of $I_1(i_1)\times I_2(i_2)$.
Next, we define $I_1(i_1-1)\times I_2(i_2+1)$ by taking 
\[
z_1(i_1-1)\in \{z\in [-\lambda_1, \lambda_1]: 0<z-z_1(i_1)< \frac{1}{2\lambda_1X}, z+z_2(i_2+1) \in J\}
\]
and then taking
\[
z_2(i_2+2)\in \{z\in [-\lambda_2, \lambda_2]:  -\frac{1}{2\lambda_2X} <z-z_2(i_2+1)<0 , z+z_1(i_1-1) \notin J\}.
\]
Then all vertices of $I_1(i_1-1)\times I_2(i_2+1)$ satisfies \eqref{eq:unbounded-2}. 
See \Cref{fig:graph1} for an illustration of how the boxes can be  constructed inductively.
By \eqref{eq:V_2_f} and \eqref{eq:unbounded}, we conclude that $V^{(2)}(f)\gg X$. 

One can proceed similarly to show  $V^{(2)}(f)\gg X$ if one of $\lambda_1, \lambda_2$ is less than 0. This shows $V^*(f)$ can be arbitrarily large, and thus completes the proof of the proposition. 

\begin{figure}[h!]
    \centering
    \includegraphics[width=0.35\textwidth]{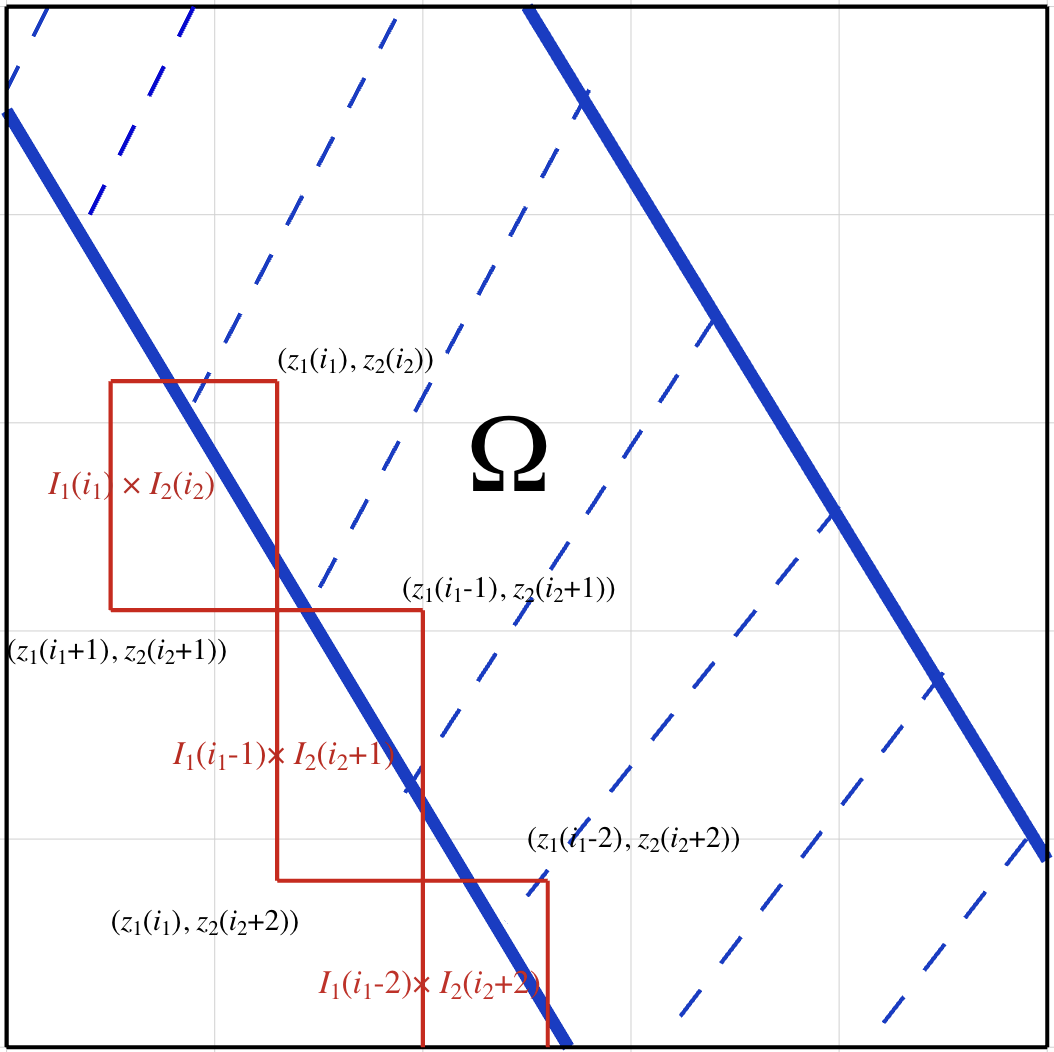}
    \caption{Inductively constructing the boxes $\{I_1(i_1)\times I_2(i_2)\}_{i_1, i_2}$.}
    \label{fig:graph1}
\end{figure}
\end{proof}

\end{appendix}
\bibliographystyle{amsplain}

\bibliography{HamdarWang_JointSatoTateLaws1}

\end{document}